\documentclass{article}

\usepackage{mathptmx}

\usepackage{geometry}
\usepackage{amssymb}
\usepackage{latexsym, amsmath, amscd,amsthm}
\usepackage{graphicx}
\usepackage[percent]{overpic}
\usepackage{units}
\usepackage{xcolor}
\definecolor{linkblue}{HTML}{003d73}
\definecolor{linkgreen}{HTML}{006161}
\definecolor{linkred}{HTML}{a11950}
\usepackage[pagebackref]{hyperref}
\hypersetup{
	pdftitle={All Prime Knots Through 10 Crossings Have Superbridge Index $\leq 5$},
	pdfauthor={Clayton Shonkwiler},
	pdfsubject={knot theory},
	pdfkeywords={knots, random polygons, superbridge index, knot invariants},
	colorlinks=true,
	linkcolor=linkblue,
	citecolor=linkgreen,
	urlcolor=linkred
}
\PassOptionsToPackage{caption=false,labelformat=empty}{subfig}
\usepackage[lofdepth]{subfig}
\usepackage[export]{adjustbox}
\usepackage{algorithm}
\usepackage{algorithmicx}
\usepackage{algpseudocode}
\usepackage{booktabs}
\usepackage{authblk}

\usepackage{supertabular,multicol,ifthen,multirow}

\usepackage{array}
\newcolumntype{L}[1]{>{\raggedright\let\newline\\\arraybackslash\hspace{0pt}}m{#1}}
\newcolumntype{C}[1]{>{\centering\let\newline\\\arraybackslash\hspace{0pt}}m{#1}}
\newcolumntype{R}[1]{>{\raggedleft\let\newline\\\arraybackslash\hspace{0pt}}m{#1}}

\usepackage{cite}
\usepackage[nameinlink]{cleveref}

\makeatletter
\let\mcnewpage=\newpage
\newcommand{\TrickSupertabularIntoMulticols}{%
\renewcommand\newpage{%
    \if@firstcolumn%
        \hrule width\linewidth height0pt%
            \columnbreak%
        \else%
          \mcnewpage%
        \fi%
}%
}
\makeatother

\newtheorem{theorem}{Theorem}

\newtheorem{corollary}[theorem]{Corollary}

\theoremstyle{definition}

\newtheorem{conjecture}[theorem]{Conjecture}

\newcommand{\R}{\mathbb{R}}

\newcommand{\stick}{\operatorname{stick}}
\newcommand{\eqstick}{\operatorname{eqstick}}
\newcommand{\bridge}{\operatorname{b}}
\newcommand{\superbridge}{\operatorname{sb}}
\newcommand{\crossing}{\operatorname{cr}}

\newenvironment{coordinates}[1]{
	\nobreak\vfil\penalty0\vfilneg\vtop\bgroup
	\begin{center} \begin{normalsize} #1 \end{normalsize} \end{center} 

		\ttfamily \begin{tiny}\begin{center}
	}{
		\end{center}\end{tiny}\par
		\xdef\tpd{\the\prevdepth}\egroup\prevdepth=\tpd
	}

\title{All Prime Knots Through 10 Crossings Have Superbridge Index $\leq 5$}
\author{Clayton Shonkwiler}
\affil{Department of Mathematics, Colorado State University, Fort Collins, CO}
\date{}

\begin{document}

\maketitle

\begin{abstract}
	This paper gives new upper bounds on the stick numbers of the knots $9_{18}$, $10_{18}$, $10_{58}$, $10_{66}$, $10_{68}$, $10_{80}$, $10_{82}$, $10_{84}$, $10_{93}$, $10_{100}$, and $10_{152}$, as well as on the equilateral stick number of $10_{79}$. These bounds imply that the knots $10_{58}$, $10_{66}$, and $10_{80}$ have superbridge index $\leq 5$, completing the project of showing that no prime knots through 10 crossings can have superbridge index larger than 5. The current best bounds on stick number and superbridge index for prime knots through 10 crossings are given in \Cref{sec:table}.
	
\end{abstract}

The stick number and the superbridge index are two geometric knot invariants which have proven quite difficult to compute. For example, the stick number is known for only 35 of the 249 nontrivial knots through 10 crossings, and superbridge index for only 49~\cite{knotinfo,TomClay,Shonkwiler:2020gi}.

The \emph{stick number} $\stick[K]$ of a knot type $K$ is the minimum number of edges needed to construct a polygonal realization of $K$~\cite{Randell:1994bx}, and the \emph{equilateral stick number} $\eqstick[K]$ is the minimum number of edges needed when we require all edges of the polygonal realization to have the same length. These are elementary invariants of knots which give some measure of the geometric (rather than purely topological) complexity of a knot.

Another elementary invariant of knots is the superbridge index introduced by Kuiper~\cite{Kuiper:1987ki}. The \emph{superbridge number} $\superbridge(\gamma)$ of a (tamely embedded) closed curve $\gamma$ is the maximum number of local maxima of any projection of $\gamma$ to a line, and the \emph{superbridge index} $\superbridge[K]$ of a knot type $K$ is the minimum superbridge number of any realization of $K$.

The key relationship between stick number and superbridge index is due to Jin~\cite{Jin:1997da}: $\superbridge[K] \leq \frac{1}{2} \stick[K]$, because there can't be more critical points in any projection of a polygonal curve to a line than there are vertices along the curve.

While there are a few general bounds on stick number~\cite{Negami:1991gb,Huh:2011co,Calvo:2001gv,Huh:2011gp}, the typical way of producing upper bounds on the stick number of a particular knot type is by finding examples~\cite{Meissen:1998wu,Millett:1994fo,Millett:2000fe,Calvo:1998kr,Scharein:1998tu,Rawdon:2002wj,Eddy:2019uj,TomClay}. In turn, such a bound on stick number provides, via Jin's result, a bound on superbridge index, though in general this bound can be arbitrarily bad: while there are only finitely many knots $K$ with $\stick[K] \leq n$ for any $n$~\cite{Calvo:2001gv,Negami:1991gb} (for example, there are exactly 11 knots with stick number $\leq 8$~\cite{Calvo:1998uj,Calvo:2002iq}), there are infinitely many knots with superbridge index equal to 4 (including all $(2,p)$-torus knots~\cite{Kuiper:1987ki}).

The goal here is to provide new upper bounds on the stick numbers of 12 of the 9- and 10-crossing knots:

\begin{figure}[t]
	\centering
		\subfloat[$9_{18}$]{\includegraphics[height=2in]{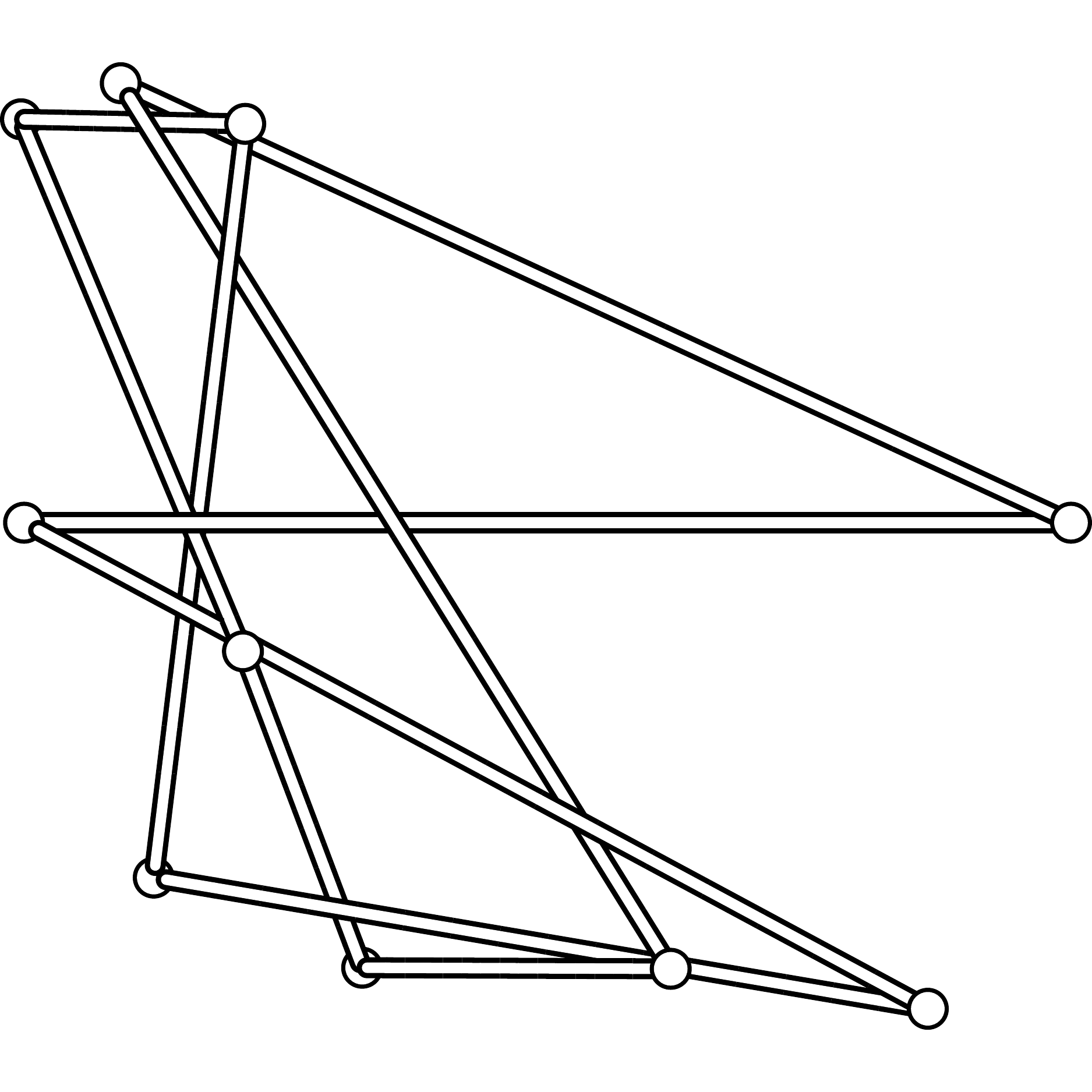}}\qquad \qquad
		\subfloat[$10_{79}$]{\includegraphics[height=2in]{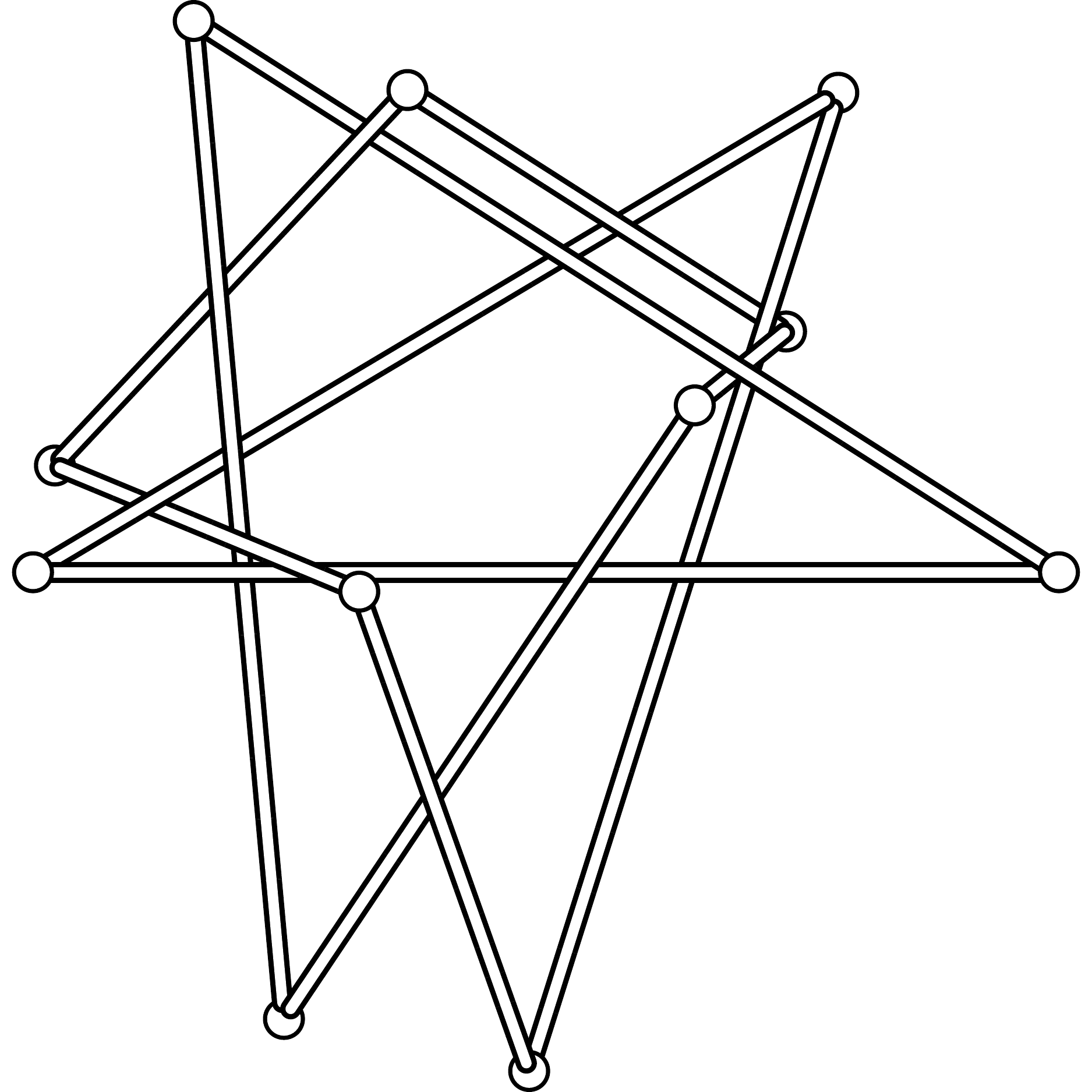}}
	\caption{An equilateral 10-stick $9_{18}$ and an equilateral 11-stick $10_{79}$. The vertex coordinates are given in \Cref{sec:coords}, and the knots are shown in orthographic perspective, viewed from the direction of the positive $z$-axis.}
	\label{fig:examples}
\end{figure}

\begin{theorem}\label{thm:main}
	The knots $9_{18}$, $10_{18}$, $10_{68}$, $10_{82}$, $10_{84}$, $10_{93}$, $10_{100}$, and $10_{152}$ have stick number and equilateral stick number $\leq 10$ and the knots $10_{58}$, $10_{66}$, $10_{79}$, and $10_{80}$ have stick number and equilateral stick number $\leq 11$.
\end{theorem}

\begin{proof}
	The result follows by exhibiting polygonal examples of each of the knots with equal-length edges. Coordinates of these realizations are given in \Cref{sec:coords} along with pictures of each knot; the pictures of the 10-stick $9_{18}$ and the 11-stick $10_{79}$ are reproduced in \Cref{fig:examples}. These examples were found among 130 billion random 10- and 11-gons generated in tight confinement using the algorithms and code from~\cite{TomClay,stick-knot-gen}; this took approximately 35,000 core-hours on a quad-core machine with Intel Xeon E5 processors. More complete analysis of this data will appear elsewhere.
	
	Since these examples are represented by the double-precision floating point coordinates of their vertices, they are only approximately equilateral. However, a result of Millett and Rawdon~\cite{Millett:2003kl} guarantees the existence of a true equilateral realization of a knot provided the edgelengths $L_1, \dots , L_n \approx 1$ and minimum distance $\mu$ between non-adjacent edges in an approximately equilateral realization satisfy 
	\[
		|L_i - 1| < \min\left\{\frac{\mu}{n}, \frac{\mu^2}{4}\right\}
	\]
	for all $i=1,\dots , n$. For all of the knots recorded in \Cref{sec:coords} we have
	\[
		|L_i - 1| < 10^{-2.96}\min\left\{\frac{\mu}{n}, \frac{\mu^2}{4}\right\},
	\]
	so the Millett--Rawdon criterion is easily satisfied and these examples also provide upper bounds on equilateral stick number.
\end{proof}

For all knots mentioned in the theorem except $10_{79}$, the stick number bound is new. It was already known~\cite{Scharein:1998tu,knotplot} that $\stick[10_{79}] \leq 11$, but the previous best bound on equilateral stick number was $12$~\cite{Rawdon:2002wj}. Since the known bounds on stick number and equilateral stick number for this knot now agree, it can be removed from Rawdon and Scharein's~\cite{Rawdon:2002wj} list of candidates for knots with unequal stick number and equilateral stick number. Indeed, combined with prior results~\cite{TomClay,Millett:2012dd}, this implies that the only knot remaining from that list is $9_{29}$, which is known to have stick number 9~\cite{Scharein:1998tu,knotplot}, but for which the best upper bound on equilateral stick number is 10. Removing $10_{79}$ from the list would seem to slightly weaken the evidence that stick number and equilateral stick number are distinct invariants, but it remains implausible to think that these two invariants will always agree.

Using Jin's bound $\superbridge[K] \leq \frac{1}{2} \stick[K]$, \Cref{thm:main} implies that all of the knots mentioned have superbridge index $\leq 5$. This result is new for the knots $10_{58}$, $10_{66}$, and $10_{80}$:

\begin{corollary}\label{cor:sb}
	The knots $10_{58}$, $10_{66}$, and $10_{80}$ have superbridge index $\leq 5$. In particular, this implies that all prime knots through 10 crossings have superbridge index $\leq 5$.
\end{corollary}

\begin{proof}
	As mentioned, the first sentence is an immediate consequence of \Cref{thm:main} and Jin's bound. The second sentence follows because all other 10-crossing knots were previously known to have superbridge index $\leq 5$. Indeed, \Cref{thm:main} implies that, with at most two possible exceptions, all 10-crossing knots have stick number $\leq 11$. The two exceptions are $10_{37}$ and $10_{76}$, both of which have been realized with 12 sticks~\cite{Rawdon:2002wj,TomClay}. While Jin's bound does not imply these knots have superbridge index $\leq 5$, such a bound has recently been proved (by very different methods) for both knots: since $10_{37}$ is a 2-bridge knot, the result $\superbridge[10_{37}] \leq 5$ is a consequence of the bound $\superbridge[K] \leq 3\bridge[K]-1$  proved by Adams et al.~\cite{Adams:2020vm}, where $\bridge[K]$ is the bridge index of the knot $K$;\footnote{For completeness, an explicit realization of $10_{37}$ with superbridge number equal to 5 is given in \Cref{sec:10_37}.} the result $\superbridge[10_{76}]\leq 5$ was proved in prior work~\cite{Shonkwiler:2020gi} by finding an explicit 12-stick realization of $10_{76}$ with superbridge number equal to 5.
\end{proof}

\Cref{cor:sb}, together with what was already known about the superbridge index of knots through 10 crossings (as recorded in \Cref{sec:table}), provides modest evidence for the following conjecture:

\begin{conjecture}\label{conj}
	For a knot $K$ with crossing number $\crossing[K] \geq 7$, 
	\begin{equation}\label{eq:conj}
		\superbridge[K] \leq \left\lceil \frac{\crossing[K]}{2} \right\rceil.
	\end{equation}
\end{conjecture}

This is somewhat stronger than the known bound $\superbridge[K] \leq \frac{3}{4}(\crossing[K]+1)$, which follows from Jin's bound and Huh and Oh's~\cite{Huh:2011co} result $\stick[K] \leq \frac{3}{2}(\crossing[K]+1)$.

\Cref{conj} holds for the one infinite family of knots for which superbridge index is known, namely the torus knots $T_{p,q}$. Assuming $2\leq p < q$, Kuiper~\cite{Kuiper:1987ki} proved that $\superbridge[T_{p,q}] = \min\{2p,q\}$, Murasugi~\cite{murasugi_braid_1991} showed that $\crossing[T_{p,q}] = q(p-1)$, and it is straightforward to verify that these quantities satisfy~\eqref{eq:conj} so long as $7 \leq \crossing[T_{p,q}] = q(p-1)$.

\subsection*{Acknowledgments} 
\label{sub:acknowledgments}

Thanks, as always, to Jason Cantarella and Tom Eddy for their mathematical and computational knowledge and to Allison Moore and Chuck Livingston for maintaining KnotInfo~\cite{knotinfo}. This work was partially supported by grants from the Simons Foundation (\#354225 and \#709150).

\clearpage 

\appendix

\section{Stick Number and Superbridge Index Bounds}\label{sec:table}

Bounds on stick number and superbridge index for prime knots through 10 crossings are given below, including references for where these results were proved. If an exact value is not known, the possible values, as determined by known upper and lower bounds, are given in the form of an interval; e.g., the entry $[9,11]$ for $\stick[9_6]$ means that $9 \leq \stick[9_6] \leq 11$.

The citations given in the table are mostly references for the upper bounds, since the lower bounds always come from the same sources: for unknown stick numbers the lower bound always comes from Calvo's characterization of knots with stick number $\leq 8$~\cite{Calvo:1998uj,Calvo:2002iq} and the lower bound on superbridge index always comes from either Kuiper's result $\bridge[K]<\superbridge[K]$ for nontrivial knots~\cite{Kuiper:1987ki} or Jeon and Jin's characterization of possible 3-superbridge knots~\cite{Jeon:2002gm}. 

The bounds for stick number and equilateral stick number agree for all knots in the table except $9_{29}$, which has $\stick[9_{29}]=9$ and $\eqstick[9_{29}] \leq 10$. This stick number is marked by an asterisk in the table as a reminder.

\Crefname{theorem}{Thm.}{Thms.}
\Crefname{corollary}{Cor.}{Cors.}

\setlength{\tabcolsep}{6pt}

\begin{center}
\begin{multicols*}{2}

\begin{tabular*}{.35\textwidth}{llll} \label{tab:stick numbers}
\!\!$K$ & $\stick[K]$ & $\superbridge[K]$ \\
\midrule
$ 0_{1} $ & 3 & 1 & \\
$ 3_{1} $ & 6 & 3 & \cite{Randell:1994bx, Kuiper:1987ki} \\
$ 4_{1} $ & 7 & 3 & \cite{Jin:1997da, Randell:1994bx} \\
$ 5_{1} $ & 8 & 4 & \cite{Negami:1991gb, Meissen:1998wu, Kuiper:1987ki} \\
$ 5_{2} $ & 8 & $[3,4]$ & \cite{Jin:1997da} \\
$ 6_{1} $ & 8 & $[3,4]$ & \cite{Negami:1991gb, Meissen:1998wu, Jin:1997da} \\
$ 6_{2} $ & 8 & $[3,4]$ & \cite{Negami:1991gb, Meissen:1998wu, Jin:1997da} \\
$ 6_{3} $ & 8 & $[3,4]$ & \cite{Negami:1991gb, Meissen:1998wu, Jin:1997da} \\
\midrule
$ 7_{1} $ & 9 & 4 & \cite{Kuiper:1987ki}\\
$ 7_{2} $ & 9 & $[3,4]$ & \cite{Meissen:1998wu, Jin:1997da}\\
$ 7_{3} $ & 9 & $[3,4]$ & \cite{Meissen:1998wu, Jin:1997da}\\
$ 7_{4} $ & 9 & $[3,4]$ & \cite{Meissen:1998wu, Jin:1997da}\\
$ 7_{5} $ & 9 & 4 & \cite{Meissen:1998wu, Jin:1997da}\\
$ 7_{6} $ & 9 & 4 & \cite{Meissen:1998wu, Jin:1997da}\\
$ 7_{7} $ & 9 & 4 & \cite{Meissen:1998wu, Jin:1997da}\\
\midrule
$ 8_{1} $ & $[9,10]$ & 4 & \cite{Rawdon:2002wj, Shonkwiler:2020gi} \\
$ 8_{2} $ & $[9,10]$ & 4 & \cite{Rawdon:2002wj, Shonkwiler:2020gi} \\
$ 8_{3} $ & $[9,10]$ & 4 & \cite{Rawdon:2002wj, Shonkwiler:2020gi} \\
$ 8_{4} $ & $[9,10]$ & $[3,4]$ & \cite{Rawdon:2002wj, Shonkwiler:2020gi} \\
$ 8_{5} $ & $[9,10]$ & 4 & \cite{Rawdon:2002wj, Shonkwiler:2020gi} \\
$ 8_{6} $ & $[9,10]$ & 4 & \cite{Rawdon:2002wj, Shonkwiler:2020gi} \\
$ 8_{7} $ & $[9,10]$ & 4 & \cite{Rawdon:2002wj, Shonkwiler:2020gi} \\
$ 8_{8} $ & $[9,10]$ & 4 & \cite{Rawdon:2002wj, Shonkwiler:2020gi} \\
$ 8_{9} $ & $[9,10]$ & $[3,4]$ & \cite{Rawdon:2002wj, Shonkwiler:2020gi} \\
$ 8_{10} $ & $[9,10]$ & 4 & \cite{Rawdon:2002wj, Shonkwiler:2020gi} \\
$ 8_{11} $ & $[9,10]$ & 4 & \cite{Rawdon:2002wj, Shonkwiler:2020gi} \\
$ 8_{12} $ & $[9,10]$ & 4 & \cite{Rawdon:2002wj, Shonkwiler:2020gi} \\
$ 8_{13} $ & $[9,10]$ & 4 & \cite{Rawdon:2002wj, Shonkwiler:2020gi} \\
\end{tabular*}

\columnbreak

\begin{tabular*}{.35\textwidth}{llll}
$K$ & $\stick[K]$ & $\superbridge[K]$ \\
\midrule
$ 8_{14} $ & $[9,10]$ & 4 & \cite{Rawdon:2002wj, Shonkwiler:2020gi} \\
$ 8_{15} $ & $[9,10]$ & 4 & \cite{Rawdon:2002wj, Shonkwiler:2020gi} \\
$ 8_{16} $ & 9 & 4 & \cite{Rawdon:2002wj, Jin:1997da} \\
$ 8_{17} $ & 9 & 4 & \cite{Rawdon:2002wj, Jin:1997da} \\
$ 8_{18} $ & 9 & 4 &  \cite{Calvo:2001gv, Rawdon:2002wj, Jin:1997da} \\
$ 8_{19} $ & 8 & 4 & \cite{Adams:1997gb, Jin:1997da, Millett:2012dd, Kuiper:1987ki} \\
$ 8_{20} $ & 8 & 4 & \cite{Negami:1991gb, Meissen:1998wu, Jin:1997da} \\
$ 8_{21} $ & 9 & 4 & \cite{Meissen:1998wu, Jin:1997da}\\
\midrule
$ 9_{1} $ & $[9,10]$ & 4 & \cite{Rawdon:2002wj, Kuiper:1987ki} \\
$ 9_{2} $ & $[9,10]$ & $[4,5]$ & \cite{TomClay, Jin:1997da} \\
$ 9_{3} $ & $[9,10]$ & $[4,5]$ & \cite{TomClay, Jin:1997da} \\
$ 9_{4} $ & $[9,10]$ & $[4,5]$ & \cite{Rawdon:2002wj, Jin:1997da} \\
$ 9_{5} $ & $[9,10]$ & $[4,5]$ & \cite{Rawdon:2002wj, Jin:1997da} \\
$ 9_{6} $ & $[9,11]$ & $[4,5]$ & \cite{Rawdon:2002wj, Jin:1997da} \\
$ 9_{7} $ & $[9,10]$ & 4 & \cite{Rawdon:2002wj, Shonkwiler:2020gi} \\
$ 9_{8} $ & $[9,10]$ & $[4,5]$ & \cite{Rawdon:2002wj, Jin:1997da} \\
$ 9_{9} $ & $[9,10]$ & $[4,5]$ & \cite{Rawdon:2002wj, Jin:1997da} \\
$ 9_{10} $ & $[9,10]$ & $[4,5]$ & \cite{Rawdon:2002wj, Jin:1997da} \\
$ 9_{11} $ & $[9,10]$ & $[4,5]$ & \cite{TomClay, Jin:1997da} \\
$ 9_{12} $ & $[9,10]$ & $[4,5]$ & \cite{Rawdon:2002wj, Jin:1997da} \\
$ 9_{13} $ & $[9,10]$ & $[4,5]$ & \cite{Rawdon:2002wj, Jin:1997da} \\
$ 9_{14} $ & $[9,10]$ & $[4,5]$ & \cite{Rawdon:2002wj, Jin:1997da} \\
$ 9_{15} $ & $[9,10]$ & $[4,5]$ & \cite{TomClay, Jin:1997da} \\
$ 9_{16} $ & $[9,10]$ & 4 & \cite{Rawdon:2002wj, Shonkwiler:2020gi} \\
$ 9_{17} $ & $[9,10]$ & $[4,5]$ & \cite{Rawdon:2002wj, Jin:1997da} \\
$ 9_{18} $ & $[9,10]$ & $[4,5]$ & \Cref{thm:main}, \cite{Jin:1997da} \\
$ 9_{19} $ & $[9,10]$ & $[4,5]$ & \cite{Rawdon:2002wj, Jin:1997da} \\
$ 9_{20} $ & $[9,10]$ & 4 & \cite{Rawdon:2002wj, Shonkwiler:2020gi} 
\end{tabular*}
\end{multicols*}

\begin{multicols*}{2}
\TrickSupertabularIntoMulticols

\tablefirsthead{
	$K$ & $\stick[K]$ & $\superbridge[K]$ \\
	\midrule
}
\tablehead{
	$K$ & $\stick[K]$ & $\superbridge[K]$ \\
	\midrule
}
\tablelasttail{\bottomrule}
\begin{supertabular*}{.35\textwidth}{llll}
$ 9_{21} $ & $[9,10]$ & $[4,5]$ & \cite{TomClay, Jin:1997da} \\
$ 9_{22} $ & $[9,10]$ & $[4,5]$ & \cite{Rawdon:2002wj, Jin:1997da} \\
$ 9_{23} $ & $[9,11]$ & $[4,5]$ & \cite{Rawdon:2002wj, Jin:1997da} \\
$ 9_{24} $ & $[9,10]$ & $[4,5]$ & \cite{Rawdon:2002wj, Jin:1997da} \\
$ 9_{25} $ & $[9,10]$ & $[4,5]$ & \cite{TomClay, Jin:1997da} \\
$ 9_{26} $ & $[9,10]$ & 4 & \cite{Rawdon:2002wj, Shonkwiler:2020gi} \\
$ 9_{27} $ & $[9,10]$ & $[4,5]$ & \cite{TomClay, Jin:1997da} \\
$ 9_{28} $ & $[9,10]$ & 4 & \cite{Rawdon:2002wj, Shonkwiler:2020gi} \\
$ 9_{29} $ & 9$^\ast$ & 4 & \cite{Scharein:1998tu, Jin:1997da} \\
$ 9_{30} $ & $[9,10]$ & $[4,5]$ & \cite{Rawdon:2002wj, Jin:1997da} \\
$ 9_{31} $ & $[9,10]$ & $[4,5]$ & \cite{Rawdon:2002wj, Jin:1997da} \\
$ 9_{32} $ & $[9,10]$ & 4 & \cite{Rawdon:2002wj, Shonkwiler:2020gi} \\
$ 9_{33} $ & $[9,10]$ & 4 & \cite{Rawdon:2002wj, Shonkwiler:2020gi} \\
$ 9_{34} $ & 9 & 4 & \cite{Rawdon:2002wj, Jin:1997da} \\
$ 9_{35} $ & 9 & 4 & \cite{TomClay, Jin:1997da} \\
$ 9_{36} $ & $[9,11]$ & $[4,5]$ & \cite{Rawdon:2002wj, Jin:1997da} \\
$ 9_{37} $ & $[9,10]$ & $[4,5]$ & \cite{Rawdon:2002wj, Jin:1997da} \\
$ 9_{38} $ & $[9,10]$ & $[4,5]$ & \cite{Rawdon:2002wj, Jin:1997da} \\
$ 9_{39} $ & 9 & 4 & \cite{TomClay, Jin:1997da} \\
$ 9_{40} $ & 9 & 4 & \cite{Scharein:1998tu, Jin:1997da} \\
$ 9_{41} $ & 9 & 4 & \cite{Scharein:1998tu, Jin:1997da} \\
$ 9_{42} $ & 9 & 4 & \cite{Calvo:1998kr, Jin:1997da} \\
$ 9_{43} $ & 9 & 4 & \cite{TomClay, Jin:1997da} \\
$ 9_{44} $ & 9 & 4 & \cite{Rawdon:2002wj, Jin:1997da} \\
$ 9_{45} $ & 9 & 4 & \cite{TomClay, Jin:1997da} \\
$ 9_{46} $ & 9 & 4 & \cite{Calvo:1998kr, Jin:1997da} \\
$ 9_{47} $ & 9 & 4 & \cite{Rawdon:2002wj, Jin:1997da} \\
$ 9_{48} $ & 9 & 4 & \cite{TomClay, Jin:1997da} \\
$ 9_{49} $ & 9 & 4 & \cite{Rawdon:2002wj, Jin:1997da} \\
\midrule
$ 10_{1} $ & $[9,11]$ & $[4,5]$ & \cite{Rawdon:2002wj, Jin:1997da} \\
$ 10_{2} $ & $[9,11]$ & $[4,5]$ & \cite{Rawdon:2002wj, Jin:1997da} \\
$ 10_{3} $ & $[9,11]$ & $[4,5]$ & \cite{TomClay, Jin:1997da} \\
$ 10_{4} $ & $[9,11]$ & $[4,5]$ & \cite{Rawdon:2002wj, Jin:1997da} \\
$ 10_{5} $ & $[9,11]$ & $[4,5]$ & \cite{Rawdon:2002wj, Jin:1997da} \\
$ 10_{6} $ & $[9,11]$ & $[4,5]$ & \cite{TomClay, Jin:1997da} \\
$ 10_{7} $ & $[9,11]$ & $[4,5]$ & \cite{TomClay, Jin:1997da} \\
$ 10_{8} $ & $[9,10]$ & $[4,5]$ & \cite{TomClay, Jin:1997da} \\
$ 10_{9} $ & $[9,11]$ & $[4,5]$ & \cite{Rawdon:2002wj, Jin:1997da} \\
$ 10_{10} $ & $[9,11]$ & $[4,5]$ & \cite{TomClay, Jin:1997da} \\
$ 10_{11} $ & $[9,11]$ & $[4,5]$ & \cite{Rawdon:2002wj, Jin:1997da} \\
$ 10_{12} $ & $[9,11]$ & $[4,5]$ & \cite{Rawdon:2002wj, Jin:1997da} \\
$ 10_{13} $ & $[9,11]$ & $[4,5]$ & \cite{Rawdon:2002wj, Jin:1997da} \\
$ 10_{14} $ & $[9,11]$ & $[4,5]$ & \cite{Rawdon:2002wj, Jin:1997da} \\
$ 10_{15} $ & $[9,11]$ & $[4,5]$ & \cite{TomClay, Jin:1997da} \\
$ 10_{16} $ & $[9,10]$ & $[4,5]$ & \cite{TomClay, Jin:1997da} \\
$ 10_{17} $ & $[9,10]$ & $[4,5]$ & \cite{TomClay, Jin:1997da} \\
$ 10_{18} $ & $[9,10]$ & $[4,5]$ & \Cref{thm:main}, \cite{Jin:1997da} \\
$ 10_{19} $ & $[9,11]$ & $[4,5]$ & \cite{Rawdon:2002wj, Jin:1997da} \\
$ 10_{20} $ & $[9,11]$ & $[4,5]$ & \cite{TomClay, Jin:1997da} \\
$ 10_{21} $ & $[9,11]$ & $[4,5]$ & \cite{TomClay, Jin:1997da} \\
$ 10_{22} $ & $[9,11]$ & $[4,5]$ & \cite{TomClay, Jin:1997da} \\
$ 10_{23} $ & $[9,11]$ & $[4,5]$ & \cite{TomClay, Jin:1997da} \\
$ 10_{24} $ & $[9,11]$ & $[4,5]$ & \cite{TomClay, Jin:1997da} \\
$ 10_{25} $ & $[9,11]$ & $[4,5]$ & \cite{Rawdon:2002wj, Jin:1997da} \\
$ 10_{26} $ & $[9,11]$ & $[4,5]$ & \cite{TomClay, Jin:1997da} \\
$ 10_{27} $ & $[9,11]$ & $[4,5]$ & \cite{Rawdon:2002wj, Jin:1997da} \\
$ 10_{28} $ & $[9,11]$ & $[4,5]$ & \cite{TomClay, Jin:1997da} \\
$ 10_{29} $ & $[9,11]$ & $[4,5]$ & \cite{Rawdon:2002wj, Jin:1997da} \\
$ 10_{30} $ & $[9,11]$ & $[4,5]$ & \cite{TomClay, Jin:1997da} \\
$ 10_{31} $ & $[9,11]$ & $[4,5]$ & \cite{TomClay, Jin:1997da} \\
$ 10_{32} $ & $[9,11]$ & $[4,5]$ & \cite{Rawdon:2002wj, Jin:1997da} \\
$ 10_{33} $ & $[9,11]$ & $[4,5]$ & \cite{Rawdon:2002wj, Jin:1997da} \\
$ 10_{34} $ & $[9,11]$ & $[4,5]$ & \cite{TomClay, Jin:1997da} \\
$ 10_{35} $ & $[9,11]$ & $[4,5]$ & \cite{TomClay, Jin:1997da} \\
$ 10_{36} $ & $[9,11]$ & $[4,5]$ & \cite{Rawdon:2002wj, Jin:1997da} \\
$ 10_{37} $ & $[9,12]$ & $[4,5]$ & \cite{Rawdon:2002wj, Adams:2020vm} \\
$ 10_{38} $ & $[9,11]$ & $[4,5]$ & \cite{TomClay, Jin:1997da} \\
$ 10_{39} $ & $[9,11]$ & $[4,5]$ & \cite{TomClay, Jin:1997da} \\
$ 10_{40} $ & $[9,11]$ & $[4,5]$ & \cite{Rawdon:2002wj, Jin:1997da} \\
$ 10_{41} $ & $[9,11]$ & $[4,5]$ & \cite{Rawdon:2002wj, Jin:1997da} \\
$ 10_{42} $ & $[9,11]$ & $[4,5]$ & \cite{Rawdon:2002wj, Jin:1997da} \\
$ 10_{43} $ & $[9,11]$ & $[4,5]$ & \cite{TomClay, Jin:1997da} \\
$ 10_{44} $ & $[9,11]$ & $[4,5]$ & \cite{TomClay, Jin:1997da} \\
$ 10_{45} $ & $[9,11]$ & $[4,5]$ & \cite{Rawdon:2002wj, Jin:1997da} \\
$ 10_{46} $ & $[9,11]$ & $[4,5]$ & \cite{TomClay, Jin:1997da} \\
$ 10_{47} $ & $[9,11]$ & $[4,5]$ & \cite{TomClay, Jin:1997da} \\
$ 10_{48} $ & $[9,10]$ & $[4,5]$ & \cite{Rawdon:2002wj, Jin:1997da} \\
$ 10_{49} $ & $[9,11]$ & $[4,5]$ & \cite{Rawdon:2002wj, Jin:1997da} \\
$ 10_{50} $ & $[9,11]$ & $[4,5]$ & \cite{TomClay, Jin:1997da} \\
$ 10_{51} $ & $[9,11]$ & $[4,5]$ & \cite{TomClay, Jin:1997da} \\
$ 10_{52} $ & $[9,11]$ & $[4,5]$ & \cite{Rawdon:2002wj, Jin:1997da} \\
$ 10_{53} $ & $[9,11]$ & $[4,5]$ & \cite{TomClay, Jin:1997da} \\
$ 10_{54} $ & $[9,11]$ & $[4,5]$ & \cite{TomClay, Jin:1997da} \\
$ 10_{55} $ & $[9,11]$ & $[4,5]$ & \cite{TomClay, Jin:1997da} \\
$ 10_{56} $ & $[9,10]$ & $[4,5]$ & \cite{TomClay, Jin:1997da} \\
$ 10_{57} $ & $[9,11]$ & $[4,5]$ & \cite{TomClay, Jin:1997da} \\
$ 10_{58} $ & $[9,11]$ & $[4,5]$ & \Cref{thm:main}, \Cref{cor:sb} \\
$ 10_{59} $ & $[9,11]$ & $[4,5]$ & \cite{Rawdon:2002wj, Jin:1997da} \\
$ 10_{60} $ & $[9,11]$ & $[4,5]$ & \cite{Rawdon:2002wj, Jin:1997da} \\
$ 10_{61} $ & $[9,11]$ & $[4,5]$ & \cite{Rawdon:2002wj, Jin:1997da} \\
$ 10_{62} $ & $[9,11]$ & $[4,5]$ & \cite{TomClay, Jin:1997da} \\
$ 10_{63} $ & $[9,11]$ & $[4,5]$ & \cite{Rawdon:2002wj, Jin:1997da} \\
$ 10_{64} $ & $[9,11]$ & $[4,5]$ & \cite{TomClay, Jin:1997da} \\
$ 10_{65} $ & $[9,11]$ & $[4,5]$ & \cite{TomClay, Jin:1997da} \\
$ 10_{66} $ & $[9,11]$ & $[4,5]$ & \Cref{thm:main}, \Cref{cor:sb} \\
$ 10_{67} $ & $[9,11]$ & $[4,5]$ & \cite{Rawdon:2002wj, Jin:1997da} \\
$ 10_{68} $ & $[9,10]$ & $[4,5]$ & \Cref{thm:main}, \cite{Jin:1997da} \\
$ 10_{69} $ & $[9,11]$ & $[4,5]$ & \cite{Rawdon:2002wj, Jin:1997da} \\
$ 10_{70} $ & $[9,11]$ & $[4,5]$ & \cite{TomClay, Jin:1997da} \\
$ 10_{71} $ & $[9,11]$ & $[4,5]$ & \cite{TomClay, Jin:1997da} \\
$ 10_{72} $ & $[9,11]$ & $[4,5]$ & \cite{TomClay, Jin:1997da} \\
$ 10_{73} $ & $[9,11]$ & $[4,5]$ & \cite{TomClay, Jin:1997da} \\
$ 10_{74} $ & $[9,11]$ & $[4,5]$ & \cite{TomClay, Jin:1997da} \\
$ 10_{75} $ & $[9,11]$ & $[4,5]$ & \cite{TomClay, Jin:1997da} \\
$ 10_{76} $ & $[9,12]$ & $[4,5]$ & \cite{TomClay, Shonkwiler:2020gi} \\
$ 10_{77} $ & $[9,11]$ & $[4,5]$ & \cite{TomClay, Jin:1997da} \\
$ 10_{78} $ & $[9,11]$ & $[4,5]$ & \cite{TomClay, Jin:1997da} \\
$ 10_{79} $ & $[9,11]$ & $[4,5]$ & \cite{Jin:1997da, Scharein:1998tu} \\
$ 10_{80} $ & $[9,11]$ & $[4,5]$ & \Cref{thm:main}, \Cref{cor:sb} \\
$ 10_{81} $ & $[9,11]$ & $[4,5]$ & \cite{Rawdon:2002wj, Jin:1997da} \\
$ 10_{82} $ & $[9,10]$ & $[4,5]$ & \Cref{thm:main}, \cite{Jin:1997da} \\
$ 10_{83} $ & $[9,10]$ & $[4,5]$ & \cite{TomClay, Jin:1997da} \\
$ 10_{84} $ & $[9,10]$ & $[4,5]$ & \Cref{thm:main}, \cite{Jin:1997da} \\
$ 10_{85} $ & $[9,10]$ & $[4,5]$ & \cite{TomClay, Jin:1997da} \\
$ 10_{86} $ & $[9,11]$ & $[4,5]$ & \cite{Rawdon:2002wj, Jin:1997da} \\
$ 10_{87} $ & $[9,11]$ & $[4,5]$ & \cite{Rawdon:2002wj, Jin:1997da} \\
$ 10_{88} $ & $[9,11]$ & $[4,5]$ & \cite{Rawdon:2002wj, Jin:1997da} \\
$ 10_{89} $ & $[9,11]$ & $[4,5]$ & \cite{Rawdon:2002wj, Jin:1997da} \\
$ 10_{90} $ & $[9,10]$ & $[4,5]$ & \cite{TomClay, Jin:1997da} \\
$ 10_{91} $ & $[9,10]$ & $[4,5]$ & \cite{TomClay, Jin:1997da} \\
$ 10_{92} $ & $[9,11]$ & $[4,5]$ & \cite{Rawdon:2002wj, Jin:1997da} \\
$ 10_{93} $ & $[9,10]$ & $[4,5]$ & \Cref{thm:main}, \cite{Jin:1997da} \\
$ 10_{94} $ & $[9,10]$ & $[4,5]$ & \cite{TomClay, Jin:1997da} \\
$ 10_{95} $ & $[9,11]$ & $[4,5]$ & \cite{TomClay, Jin:1997da} \\
$ 10_{96} $ & $[9,11]$ & $[4,5]$ & \cite{Rawdon:2002wj, Jin:1997da} \\
$ 10_{97} $ & $[9,11]$ & $[4,5]$ & \cite{TomClay, Jin:1997da} \\
$ 10_{98} $ & $[9,11]$ & $[4,5]$ & \cite{Rawdon:2002wj, Jin:1997da} \\
$ 10_{99} $ & $[9,11]$ & $[4,5]$ & \cite{Rawdon:2002wj, Jin:1997da} \\
$ 10_{100} $ & $[9,10]$ & $[4,5]$ & \Cref{thm:main}, \cite{Jin:1997da} \\
$ 10_{101} $ & $[9,11]$ & $[4,5]$ & \cite{TomClay, Jin:1997da} \\
$ 10_{102} $ & $[9,10]$ & $[4,5]$ & \cite{Rawdon:2002wj, Jin:1997da} \\
$ 10_{103} $ & $[9,10]$ & $[4,5]$ & \cite{TomClay, Jin:1997da} \\
$ 10_{104} $ & $[9,10]$ & $[4,5]$ & \cite{Rawdon:2002wj, Jin:1997da} \\
$ 10_{105} $ & $[9,10]$ & $[4,5]$ & \cite{TomClay, Jin:1997da} \\
$ 10_{106} $ & $[9,10]$ & $[4,5]$ & \cite{TomClay, Jin:1997da} \\
$ 10_{107} $ & $[9,10]$ & $[4,5]$ & \cite{Scharein:1998tu, Jin:1997da} \\
$ 10_{108} $ & $[9,10]$ & $[4,5]$ & \cite{Rawdon:2002wj, Jin:1997da} \\
$ 10_{109} $ & $[9,10]$ & $[4,5]$ & \cite{Rawdon:2002wj, Jin:1997da} \\
$ 10_{110} $ & $[9,10]$ & $[4,5]$ & \cite{TomClay, Jin:1997da} \\
$ 10_{111} $ & $[9,10]$ & $[4,5]$ & \cite{TomClay, Jin:1997da} \\
$ 10_{112} $ & $[9,10]$ & $[4,5]$ & \cite{TomClay, Jin:1997da} \\
$ 10_{113} $ & $[9,10]$ & $[4,5]$ & \cite{Rawdon:2002wj, Jin:1997da} \\
$ 10_{114} $ & $[9,10]$ & $[4,5]$ & \cite{Rawdon:2002wj, Jin:1997da} \\
$ 10_{115} $ & $[9,10]$ & $[4,5]$ & \cite{TomClay, Jin:1997da} \\
$ 10_{116} $ & $[9,10]$ & $[4,5]$ & \cite{Rawdon:2002wj, Jin:1997da} \\
$ 10_{117} $ & $[9,10]$ & $[4,5]$ & \cite{TomClay, Jin:1997da} \\
$ 10_{118} $ & $[9,10]$ & $[4,5]$ & \cite{TomClay, Jin:1997da} \\
$ 10_{119} $ & $[9,10]$ & $[4,5]$ & \cite{Scharein:1998tu, Jin:1997da} \\
$ 10_{120} $ & $[9,10]$ & $[4,5]$ & \cite{Rawdon:2002wj, Jin:1997da} \\
$ 10_{121} $ & $[9,10]$ & $[4,5]$ & \cite{Rawdon:2002wj, Jin:1997da} \\
$ 10_{122} $ & $[9,10]$ & $[4,5]$ & \cite{Rawdon:2002wj, Jin:1997da} \\
$ 10_{123} $ & $[9,11]$ & $[4,5]$ & \cite{Rawdon:2002wj, Jin:1997da} \\
$ 10_{124} $ & 10 & 5 & \cite{Adams:1997gb, Jin:1997da, Kuiper:1987ki} \\
$ 10_{125} $ & $[9,10]$ & $[4,5]$ & \cite{Rawdon:2002wj, Jin:1997da} \\
$ 10_{126} $ & $[9,10]$ & $[4,5]$ & \cite{TomClay, Jin:1997da} \\
$ 10_{127} $ & $[9,10]$ & $[4,5]$ & \cite{Rawdon:2002wj, Jin:1997da} \\
$ 10_{128} $ & $[9,10]$ & $[4,5]$ & \cite{Rawdon:2002wj, Jin:1997da} \\
$ 10_{129} $ & $[9,10]$ & $[4,5]$ & \cite{Rawdon:2002wj, Jin:1997da} \\
$ 10_{130} $ & $[9,10]$ & $[4,5]$ & \cite{Rawdon:2002wj, Jin:1997da} \\
$ 10_{131} $ & $[9,10]$ & $[4,5]$ & \cite{TomClay, Jin:1997da} \\
$ 10_{132} $ & $[9,10]$ & $[4,5]$ & \cite{Rawdon:2002wj, Jin:1997da} \\
$ 10_{133} $ & $[9,10]$ & $[4,5]$ & \cite{TomClay, Jin:1997da} \\
$ 10_{134} $ & $[9,10]$ & $[4,5]$ & \cite{Rawdon:2002wj, Jin:1997da} \\
$ 10_{135} $ & $[9,10]$ & $[4,5]$ & \cite{Rawdon:2002wj, Jin:1997da} \\
$ 10_{136} $ & $[9,10]$ & $[4,5]$ & \cite{Rawdon:2002wj, Jin:1997da} \\
$ 10_{137} $ & $[9,10]$ & $[4,5]$ & \cite{TomClay, Jin:1997da} \\
$ 10_{138} $ & $[9,10]$ & $[4,5]$ & \cite{TomClay, Jin:1997da} \\
$ 10_{139} $ & $[9,10]$ & $[4,5]$ & \cite{Rawdon:2002wj, Jin:1997da} \\
$ 10_{140} $ & $[9,10]$ & $[4,5]$ & \cite{Rawdon:2002wj, Jin:1997da} \\
$ 10_{141} $ & $[9,10]$ & $[4,5]$ & \cite{Rawdon:2002wj, Jin:1997da} \\
$ 10_{142} $ & $[9,10]$ & $[4,5]$ & \cite{TomClay, Jin:1997da} \\
$ 10_{143} $ & $[9,10]$ & $[4,5]$ & \cite{TomClay, Jin:1997da} \\
$ 10_{144} $ & $[9,10]$ & $[4,5]$ & \cite{Rawdon:2002wj, Jin:1997da} \\
$ 10_{145} $ & $[9,10]$ & $[4,5]$ & \cite{Rawdon:2002wj, Jin:1997da} \\
$ 10_{146} $ & $[9,10]$ & $[4,5]$ & \cite{Rawdon:2002wj, Jin:1997da} \\
$ 10_{147} $ & $[9,10]$ & $[4,5]$ & \cite{Scharein:1998tu, Jin:1997da} \\
$ 10_{148} $ & $[9,10]$ & $[4,5]$ & \cite{TomClay, Jin:1997da} \\
$ 10_{149} $ & $[9,10]$ & $[4,5]$ & \cite{TomClay, Jin:1997da} \\
$ 10_{150} $ & $[9,10]$ & $[4,5]$ & \cite{Rawdon:2002wj, Jin:1997da} \\
$ 10_{151} $ & $[9,10]$ & $[4,5]$ & \cite{Rawdon:2002wj, Jin:1997da} \\
$ 10_{152} $ & $[9,10]$ & $[4,5]$ & \Cref{thm:main}, \cite{Jin:1997da} \\
$ 10_{153} $ & $[9,10]$ & $[4,5]$ & \cite{TomClay, Jin:1997da} \\
$ 10_{154} $ & $[9,11]$ & $[4,5]$ & \cite{Rawdon:2002wj, Jin:1997da} \\
$ 10_{155} $ & $[9,10]$ & $[4,5]$ & \cite{Rawdon:2002wj, Jin:1997da} \\
$ 10_{156} $ & $[9,10]$ & $[4,5]$ & \cite{Rawdon:2002wj, Jin:1997da} \\
$ 10_{157} $ & $[9,10]$ & $[4,5]$ & \cite{Rawdon:2002wj, Jin:1997da} \\
$ 10_{158} $ & $[9,10]$ & $[4,5]$ & \cite{Rawdon:2002wj, Jin:1997da} \\
$ 10_{159} $ & $[9,10]$ & $[4,5]$ & \cite{Rawdon:2002wj, Jin:1997da} \\
$ 10_{160} $ & $[9,10]$ & $[4,5]$ & \cite{Rawdon:2002wj, Jin:1997da} \\
$ 10_{161} $ & $[9,10]$ & $[4,5]$ & \cite{Rawdon:2002wj, Jin:1997da} \\
$ 10_{162} $ & $[9,10]$ & $[4,5]$ & \cite{Rawdon:2002wj, Jin:1997da} \\
$ 10_{163} $ & $[9,10]$ & $[4,5]$ & \cite{Rawdon:2002wj, Jin:1997da} \\
$ 10_{164} $ & $[9,10]$ & $[4,5]$ & \cite{TomClay, Jin:1997da} \\
$ 10_{165} $ & $[9,10]$ & $[4,5]$ & \cite{Rawdon:2002wj, Jin:1997da} \\
\end{supertabular*}
\end{multicols*}

\end{center}

\newpage

\section{Knot Images and Coordinates} 
\label{sec:coords}

The knot coordinates are normalized so that the first vertex is at the origin, the second is at $(1,0,0)$, and the third is in the $xy$-plane with positive $y$-coordinate. Each knot is shown in orthographic perspective from the direction of the positive $z$-axis. These coordinates can also be downloaded from the {\tt stick-knot-gen} project~\cite{stick-knot-gen}.

\vspace{-.2in}
\begin{multicols*}{2}

\begin{coordinates}{$\boldmath{} $ $\boldmath{9_{18}} $}

\includegraphics[height=1.55in]{9_18.pdf}

\begin{tabular}{lll}

0. & 0. & 0. \\
0.9999999999999999 & 0. & 0. \\
0.1369811385433643 & 0.5051716983067165 & 0. \\
0.36445473772481973 & -0.46837424343037093 & 0.02154207517120799 \\
-0.40824166622547475 & -0.1946531623448116 & 0.5942697915353068 \\
0.1761863718304549 & 0.05044370765670344 & -0.17927490184029685 \\
-0.0358214851573142 & 0.11220243094991508 & 0.796039677499439 \\
0.36991848186843224 & 0.4169202041594481 & -0.065657576672912 \\
-0.2483842722893944 & -0.23075441407147784 & 0.3795606766593944 \\
0.6123146506430659 & -0.7330077964576398 & 0.29626059295858787\\

\end{tabular}
\end{coordinates}
\vspace{0.1in}

\begin{coordinates}{$\boldmath{} $ $\boldmath{10_{18}} $}

\includegraphics[height=1.55in]{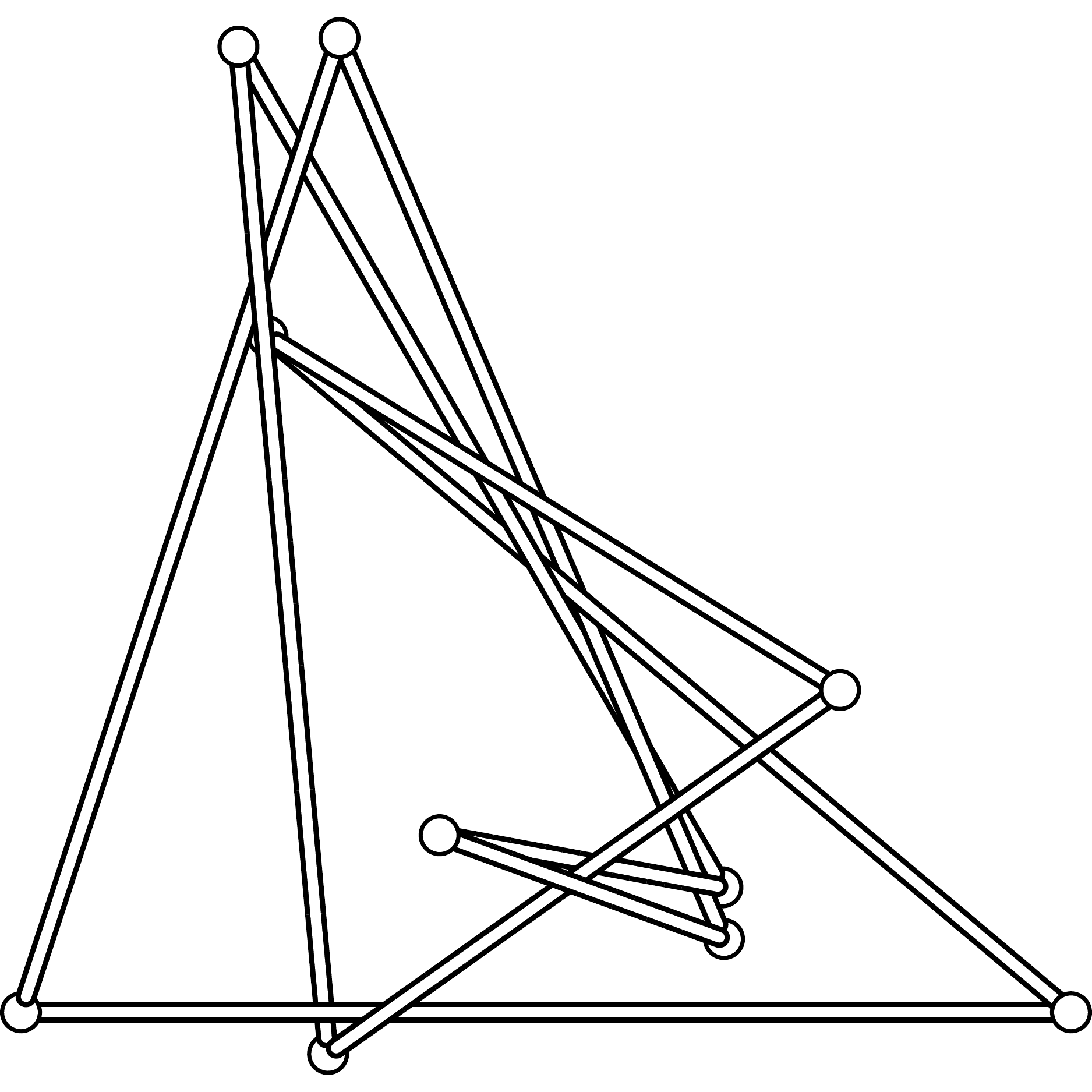}

\begin{tabular}{lll}

0. & 0. & 0. \\
0.9999999999999999 & 0. & 0. \\
0.10374904044780832 & 0.4435473114581754 & 0. \\
0.4061269411819969 & -0.1391324357591598 & 0.7543552991331917 \\
0.05220407842345304 & 0.41171347034152267 & -0.0014934919281650513 \\
1.0010147153854116 & 0.09608809195992474 & 0.01029615222744267 \\
0.06703161095626266 & -0.2530860092405562 & -0.06555243175225965 \\
0.007156648478157002 & 0.6766297479925333 & 0.29782560778886097 \\
0.6758185737598675 & 0.07818217741767258 & -0.14348124118541358 \\
0.5092073572964613 & 0.21798840253364535 & 0.8325796800522959\\

\end{tabular}
\end{coordinates}
\vspace{0.1in}

\begin{coordinates}{$\boldmath{} $ $\boldmath{10_{58}} $}

\includegraphics[height=1.55in]{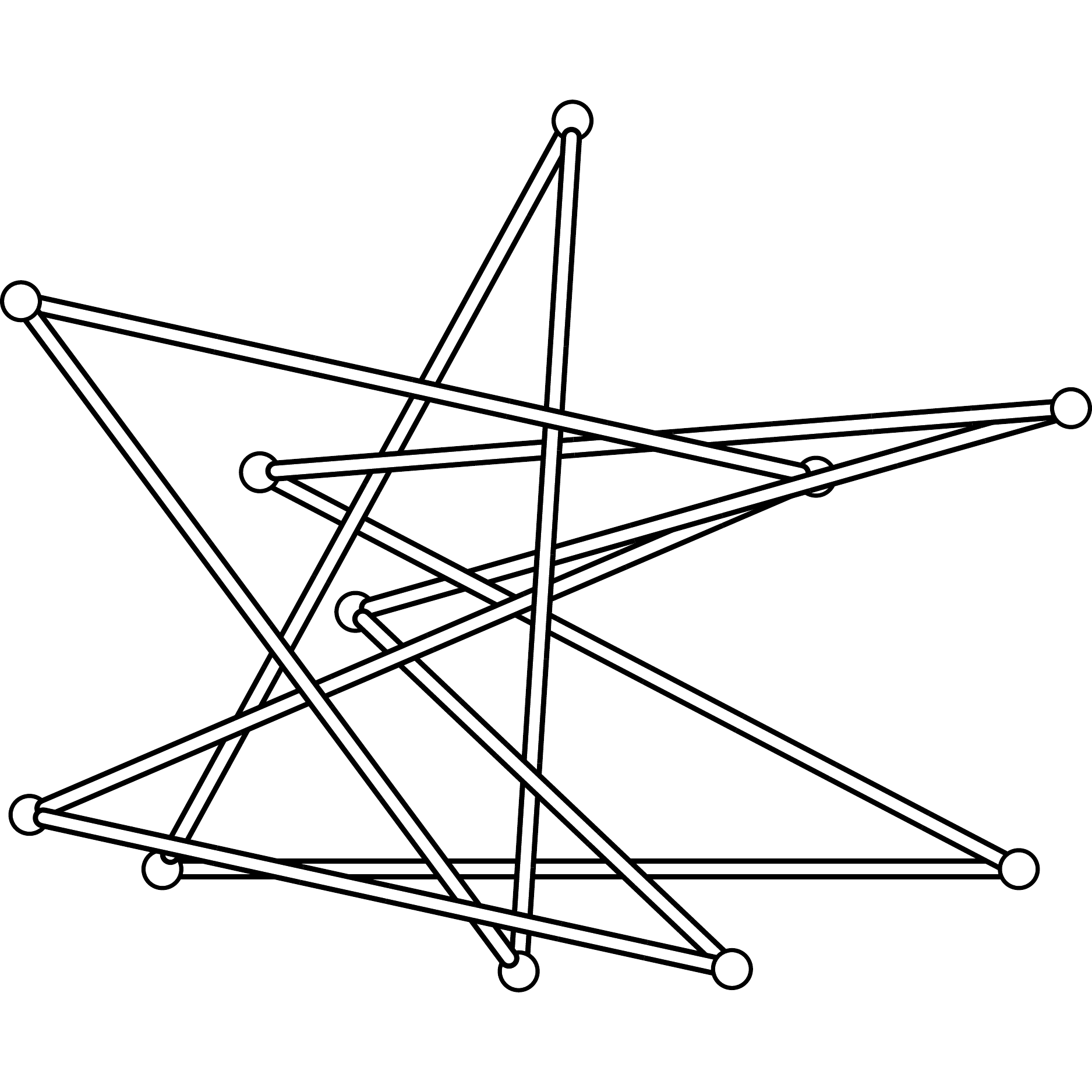}

\begin{tabular}{lll}

0. & 0. & 0. \\
0.9999999999999999 & 0. & 0. \\
0.23075419502569852 & 0.6389529650368944 & 0. \\
0.34568868227993477 & -0.34762407938349305 & 0.1159991338861572 \\
0.5854402278388353 & 0.6046839598660042 & 0.304754517340236 \\
-0.34288871608553806 & 0.39659016855951384 & -0.00330772619497593 \\
0.5747045141992497 & 0.1278600084591875 & 0.28962019119765003 \\
-0.19231721046439432 & -0.13517667062917393 & -0.2956057304327706 \\
0.3810463453997557 & 0.2597856620372011 & 0.4222096975022616 \\
0.13411123197649716 & 0.5946070064410464 & -0.487140464476194 \\
0.7640898592707 & 0.5766222162271076 & 0.2892637321079513\\

\end{tabular}
\end{coordinates}
\vspace{0.1in}

\begin{coordinates}{$\boldmath{} $ $\boldmath{10_{66}} $}

\includegraphics[height=1.55in]{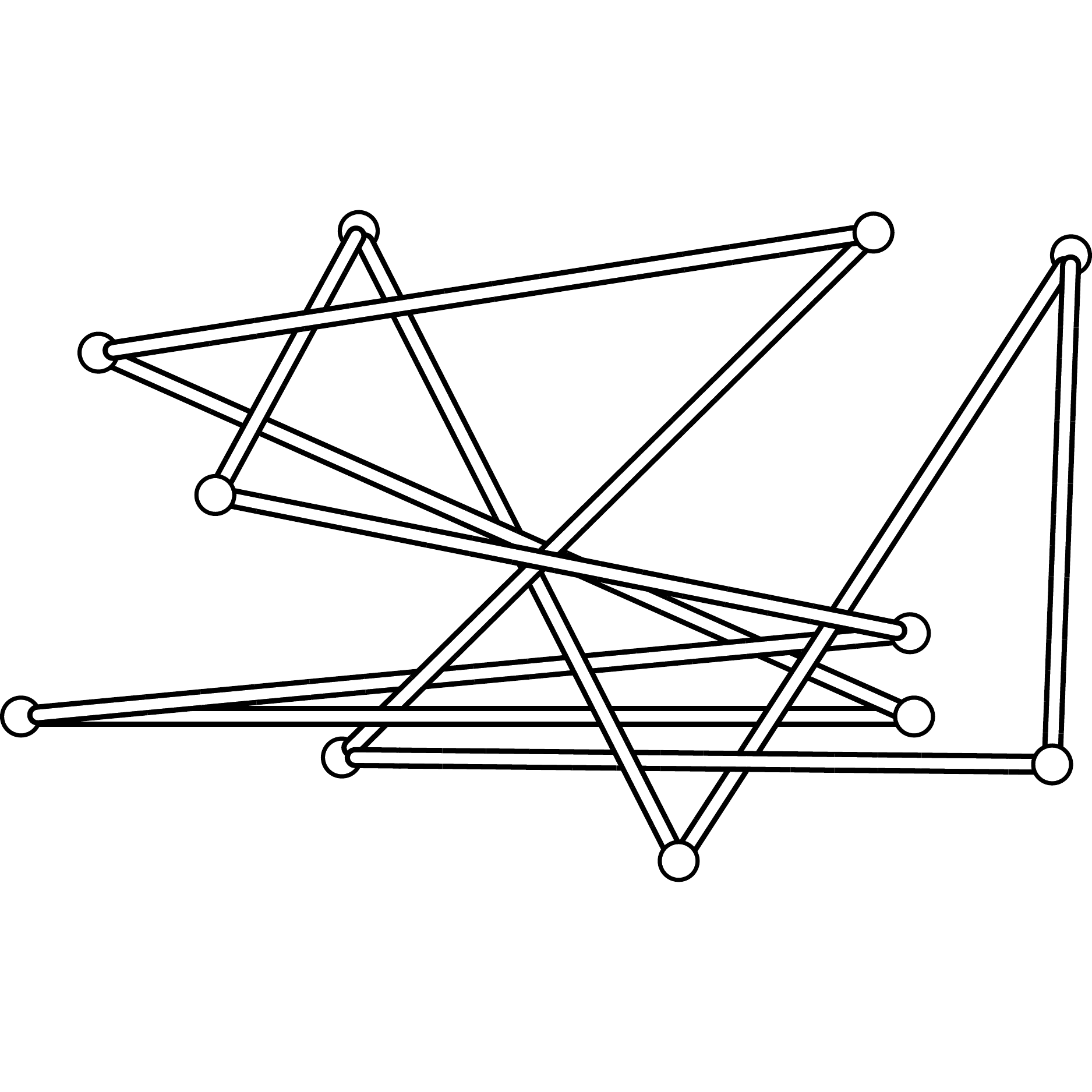}

\begin{tabular}{lll}

0. & 0. & 0. \\
0.9999999999999999 & 0. & 0. \\
0.5145101891023756 & 0.8742423253964475 & 0. \\
0.5275451799672034 & -0.12484964472859603 & -0.040562596623153806 \\
-0.13361019429763404 & 0.5339205870767882 & 0.318457392747638 \\
0.3118131683954475 & -0.3000760437290037 & 0.6441069709321883 \\
0.5621860061072727 & 0.3292462605214777 & -0.09160144923778166 \\
-0.24397323845573143 & -0.1795895608670905 & 0.21038096627250544 \\
0.5513339449909647 & 0.2862276622235819 & -0.17756335786457722 \\
-0.17814186103165766 & 0.12471725485537188 & 0.4871015477377872 \\
0.8009432011351053 & 0.32768567714920493 & 0.5011108515555394\\

\end{tabular}
\end{coordinates}
\vspace{0.1in}

\begin{coordinates}{$\boldmath{} $ $\boldmath{10_{68}} $}

\includegraphics[height=1.55in]{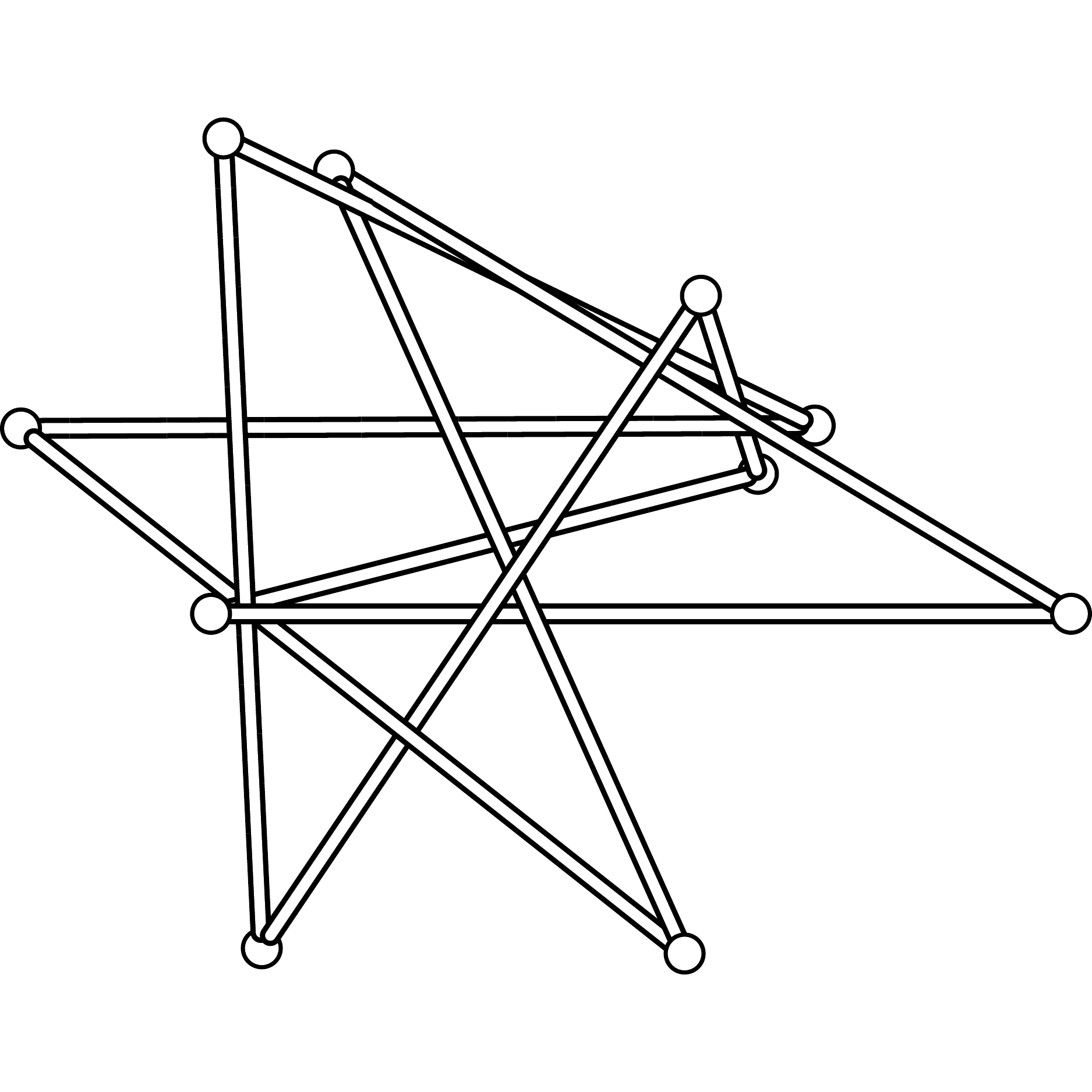}

\begin{tabular}{lll}

0. & 0. & 0. \\
0.9999999999999999 & 0. & 0. \\
0.14308910801988475 & 0.5154645702721413 & 0. \\
0.5507160581438794 & -0.3950793595184551 & 0.0689203994047325 \\
-0.220985520662934 & 0.2155209342031813 & -0.10896707787830442 \\
0.7019924653376205 & 0.21890712528211487 & -0.4938049282365256 \\
0.014394765753044353 & 0.5527134207336278 & 0.1510072619731009 \\
0.059021448275970674 & -0.38897111542175017 & -0.18251739014848975 \\
0.5697203717471668 & 0.36984460712635464 & 0.2216920961072869 \\
0.6359432500426695 & 0.16273761313218546 & -0.7543822983058428\\

\end{tabular}
\end{coordinates}
\vspace{0.1in}

\begin{coordinates}{$\boldmath{} $ $\boldmath{10_{79}} $}

\includegraphics[height=1.55in]{10_79.pdf}

\begin{tabular}{lll}

0. & 0. & 0. \\ 
0.9999999999999999 & 0. & 0. \\
0.15662097603563974 & 0.537319106245928 & 0. \\
0.2445856042154266 & -0.43538593154418503 & -0.21472571724627276 \\
0.6450217925958391 & 0.1626522315261948 & 0.4795373557647955 \\
0.734006748172326 & 0.23463544410470605 & -0.5138910990259495 \\
0.36493154829639873 & 0.47001886661877124 & 0.3852074168454255 \\
0.021330163575223243 & 0.10403072835102278 & -0.47965718010891734 \\
0.3179757846004717 & -0.018939867230503504 & 0.46738009906697064 \\
0.4835968181650272 & -0.48616249117625676 & -0.401108641743325 \\
0.7849150964041073 & 0.4673440590915645 & -0.40681423508617587\\

\end{tabular}
\end{coordinates}
\vspace{0.1in}

\begin{coordinates}{$\boldmath{} $ $\boldmath{10_{80}} $}

\includegraphics[height=1.55in]{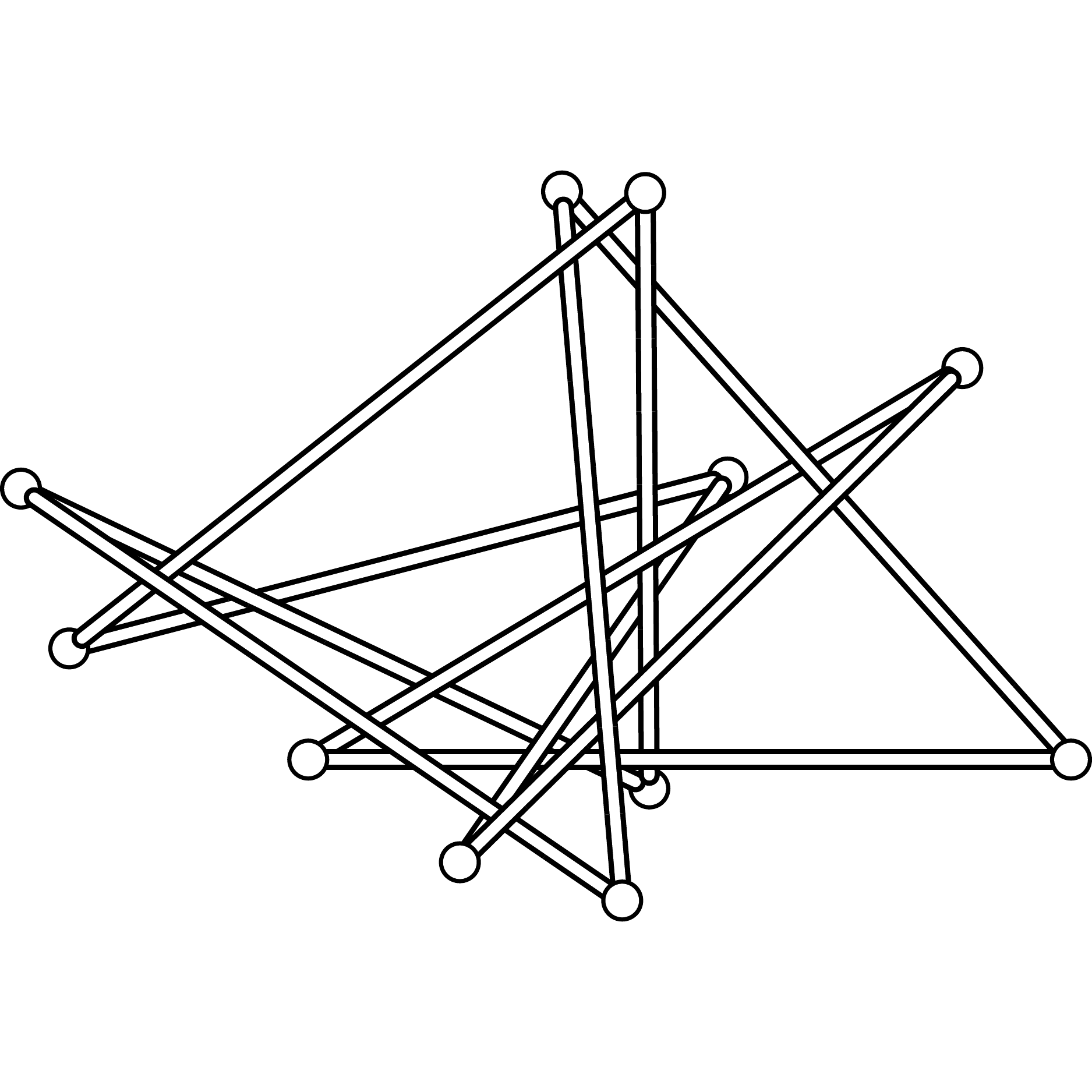}

\begin{tabular}{lll}

0. & 0. & 0. \\
0.9999999999999999 & 0. & 0. \\
0.3327312868213165 & 0.7448170677507757 & 0. \\
0.41164114869381774 & -0.18477978499933626 & 0.36003183894809804 \\
-0.37657987371245755 & 0.35524645381103787 & 0.06493985119999454 \\
0.4472049020942717 & -0.03782867814303159 & -0.3435581093681675 \\
0.44180933980825615 & 0.7428150212139177 & 0.2813949309171927 \\
-0.3132726693922105 & 0.1457679277099991 & 0.010496559714429563 \\
0.5497691607609388 & 0.36988149308037604 & -0.4421974972632168 \\
0.19881030933734742 & -0.13460231801890482 & 0.3466776316045772 \\
0.8574892644307307 & 0.5132409345720769 & -0.03599867310924573\\

\end{tabular}
\end{coordinates}
\vspace{0.1in}

\begin{coordinates}{$\boldmath{} $ $\boldmath{10_{82}} $}

\includegraphics[height=1.55in]{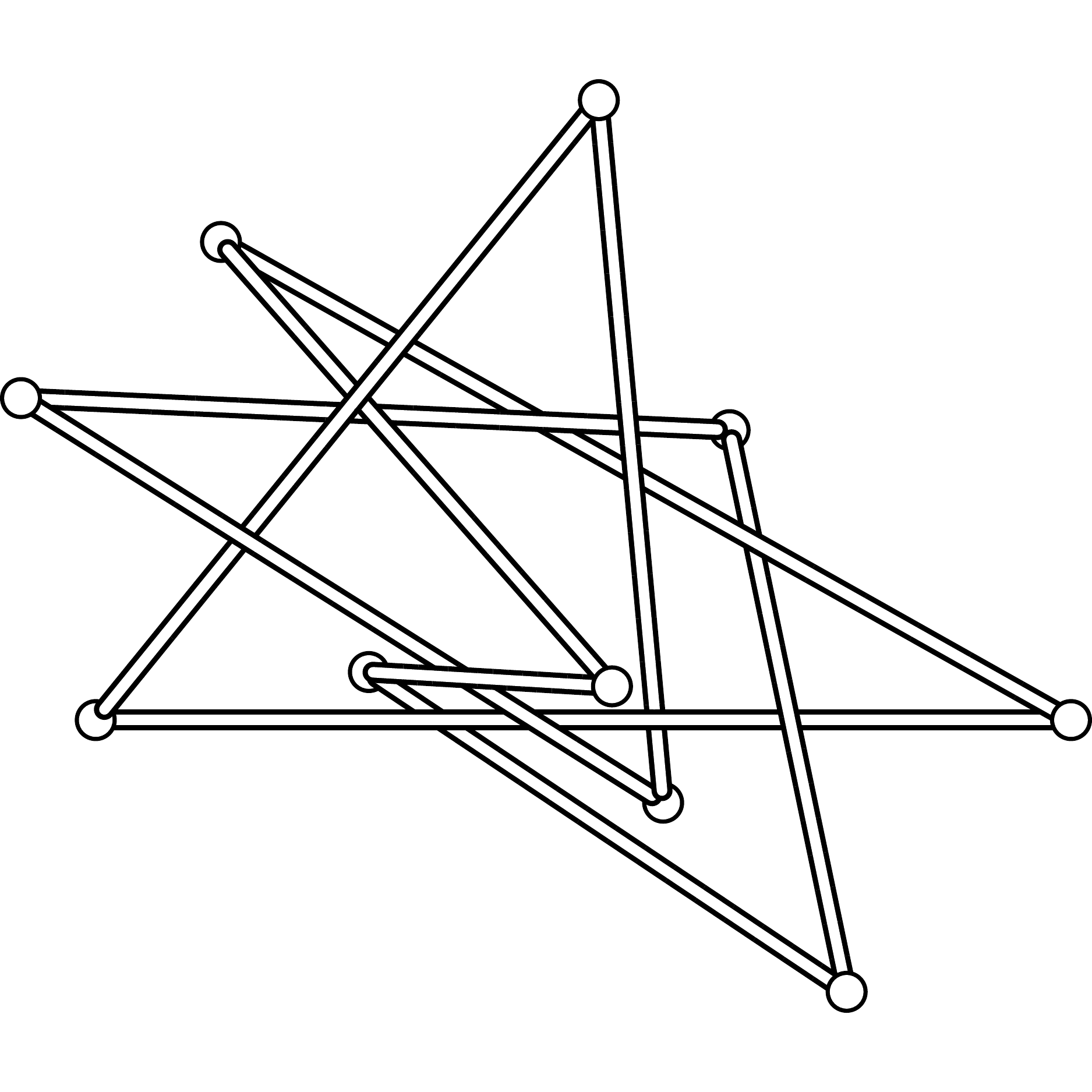}

\begin{tabular}{lll}

0. & 0. & 0. \\
0.9999999999999999 & 0. & 0. \\
0.09089527821440953 & 0.41656764736011836 & 0. \\
0.9122855066438609 & 0.04664743439303293 & -0.4341395267411148 \\
0.4293464164746975 & -0.5569166222908565 & 0.20027379646849952 \\
0.8672055548188717 & 0.28186943587667174 & -0.12332676532377572 \\
0.2256114315412598 & -0.4239913296650354 & -0.4235226357892222 \\
0.901013770664066 & -0.17005456734157295 & 0.2688268431938411 \\
0.651739366157114 & 0.46199768313726536 & -0.4649116823875408 \\
0.8761992708919832 & -0.45907319072845804 & -0.14671960756064373\\

\end{tabular}
\end{coordinates}
\vspace{0.1in}

\begin{coordinates}{$\boldmath{} $ $\boldmath{10_{84}} $}

\includegraphics[height=1.55in]{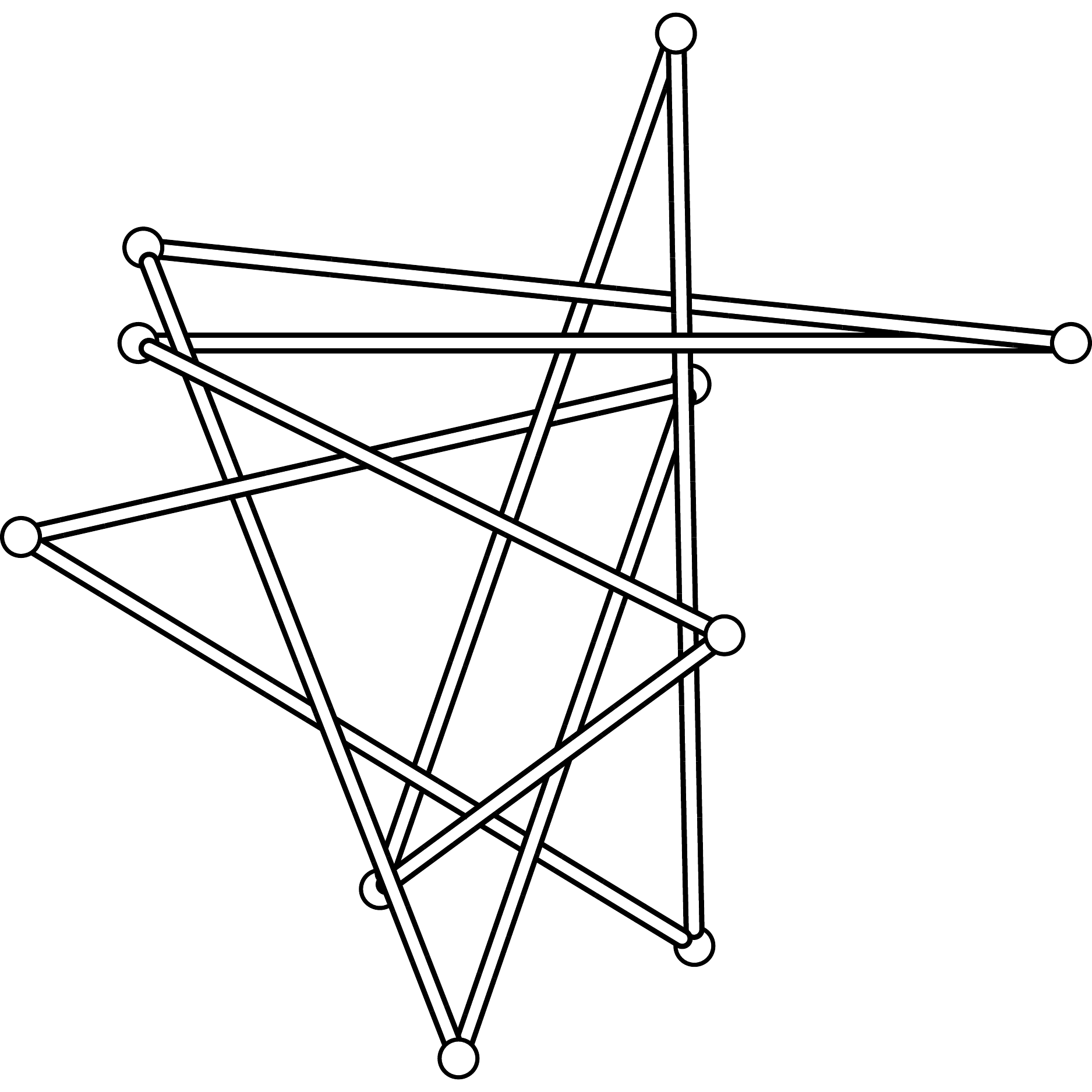}

\begin{tabular}{lll}

0. & 0. & 0. \\
0.9999999999999999 & 0. & 0. \\
0.19130320128681053 & 0.5882257115691553 & 0. \\
0.18504701394830536 & -0.36160278132829377 & 0.3127086410704792 \\
0.4007142236869289 & 0.4473066686279016 & -0.23424031122317077 \\
0.6920922472558773 & 0.29033346909917157 & 0.7094006491058975 \\
0.6101485945447912 & -0.17959763305314008 & -0.1694906962891281 \\
-0.10778017173660057 & 0.06099088207608897 & 0.483728602976893 \\
0.5896371143947019 & -0.27762434838744 & -0.14789537907522962 \\
0.7331230152034498 & 0.6547968310753483 & 0.1837709296724706\\

\end{tabular}
\end{coordinates}
\vspace{0.1in}

\begin{coordinates}{$\boldmath{} $ $\boldmath{10_{93}} $}

\includegraphics[height=1.55in]{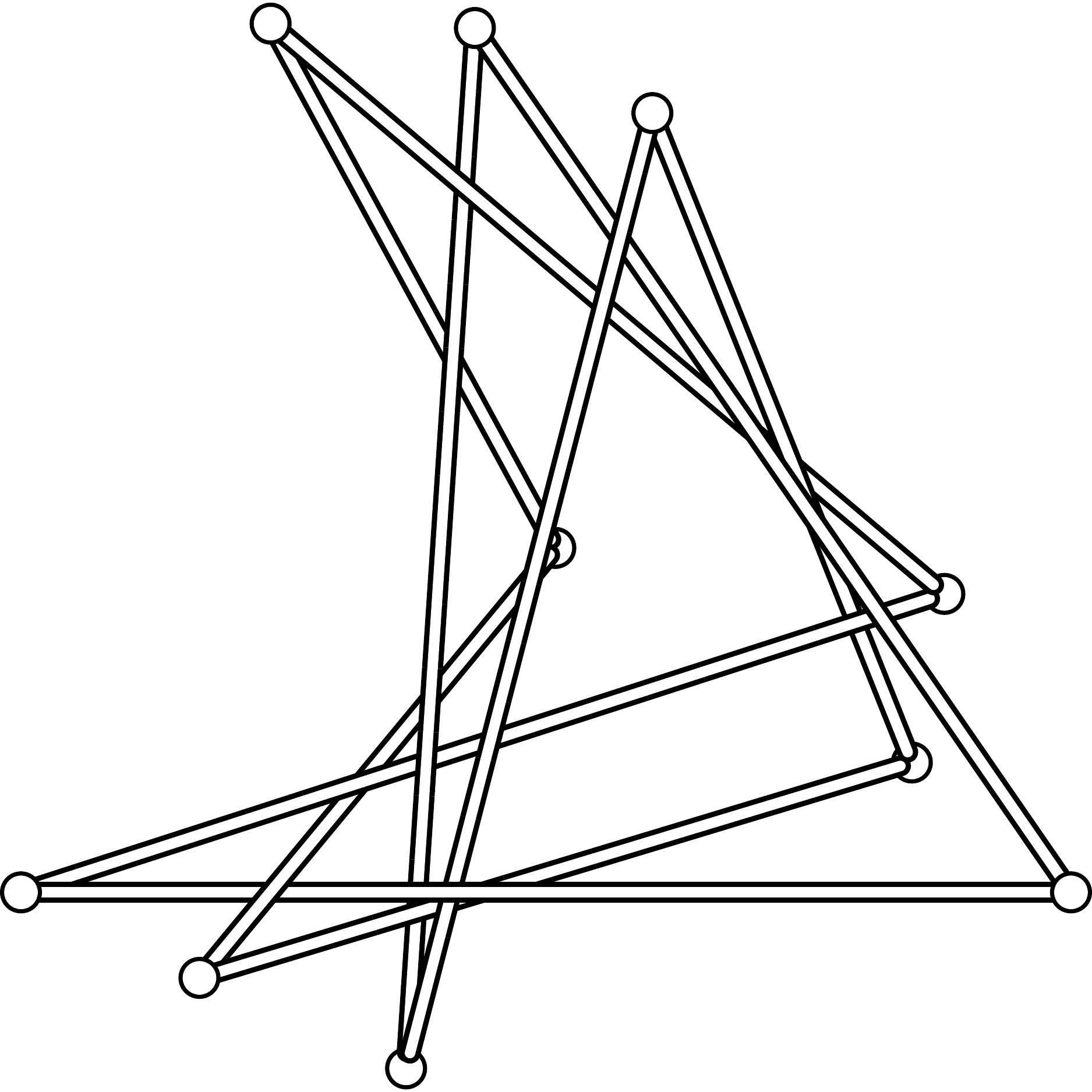}

\begin{tabular}{lll}

0. & 0. & 0. \\
0.9999999999999999 & 0. & 0. \\
0.4322831520544033 & 0.8232238945503928 & 0. \\
0.366749435669193 & -0.16760413688857084 & -0.11817421094101156 \\
0.6013038771595937 & 0.7421884288832536 & 0.22426074820224834 \\
0.8484129178261877 & 0.12369596771944935 & -0.5216643133852914 \\
0.1699818416631752 & -0.08145086977191296 & 0.18377601007487826 \\
0.5088788663970447 & 0.3280505164135769 & -0.6632525734680226 \\
0.23773536251973704 & 0.8273614792778586 & 0.1596506788458079 \\
0.8793957949879448 & 0.28424997300173205 & -0.3819227521450315\\

\end{tabular}
\end{coordinates}
\vspace{0.1in}

\begin{coordinates}{$\boldmath{} $ $\boldmath{10_{100}} $}

\includegraphics[height=1.55in]{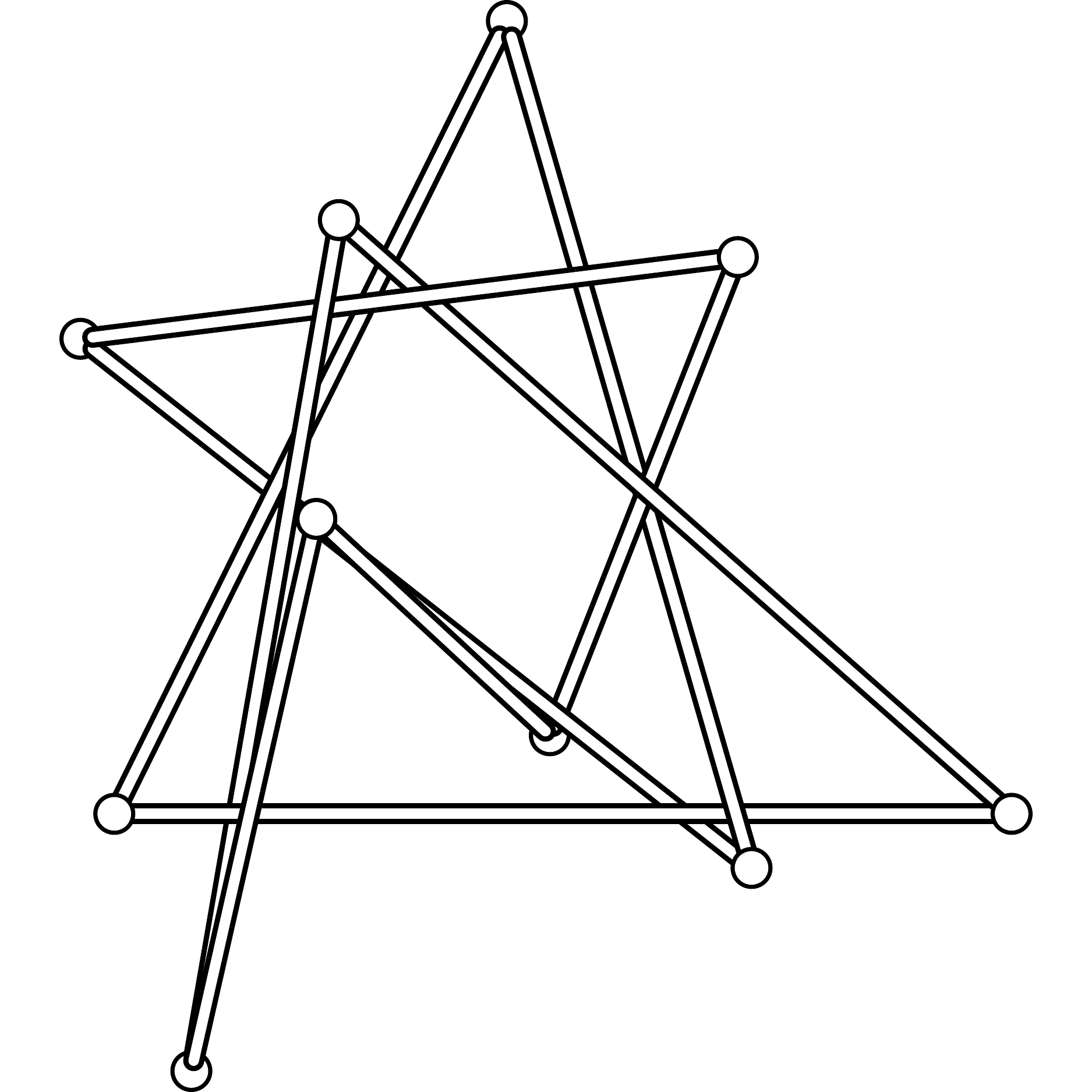}

\begin{tabular}{lll}

0. & 0. & 0. \\
0.9999999999999999 & 0. & 0. \\
0.25025317408203634 & 0.6617247894902673 & 0. \\
0.08590073137649616 & -0.2865914931622944 & -0.2714488950663621 \\
0.225196070444289 & 0.328672653046371 & 0.5044678781149572 \\
0.4853999222371724 & 0.08756968372166599 & -0.43049913397421674 \\
0.6948680476379834 & 0.6204316823737933 & 0.3893674373857395 \\
-0.03774921617740326 & 0.5293732483746156 & -0.28515484271868646 \\
0.7101850493261617 & -0.060296929991126846 & 0.01961296271974253 \\
0.4375512255320434 & 0.8835485041661305 & -0.16700588558850374\\

\end{tabular}
\end{coordinates}
\vspace{0.1in}

\begin{coordinates}{$\boldmath{} $ $\boldmath{10_{152}} $}

\includegraphics[height=1.55in]{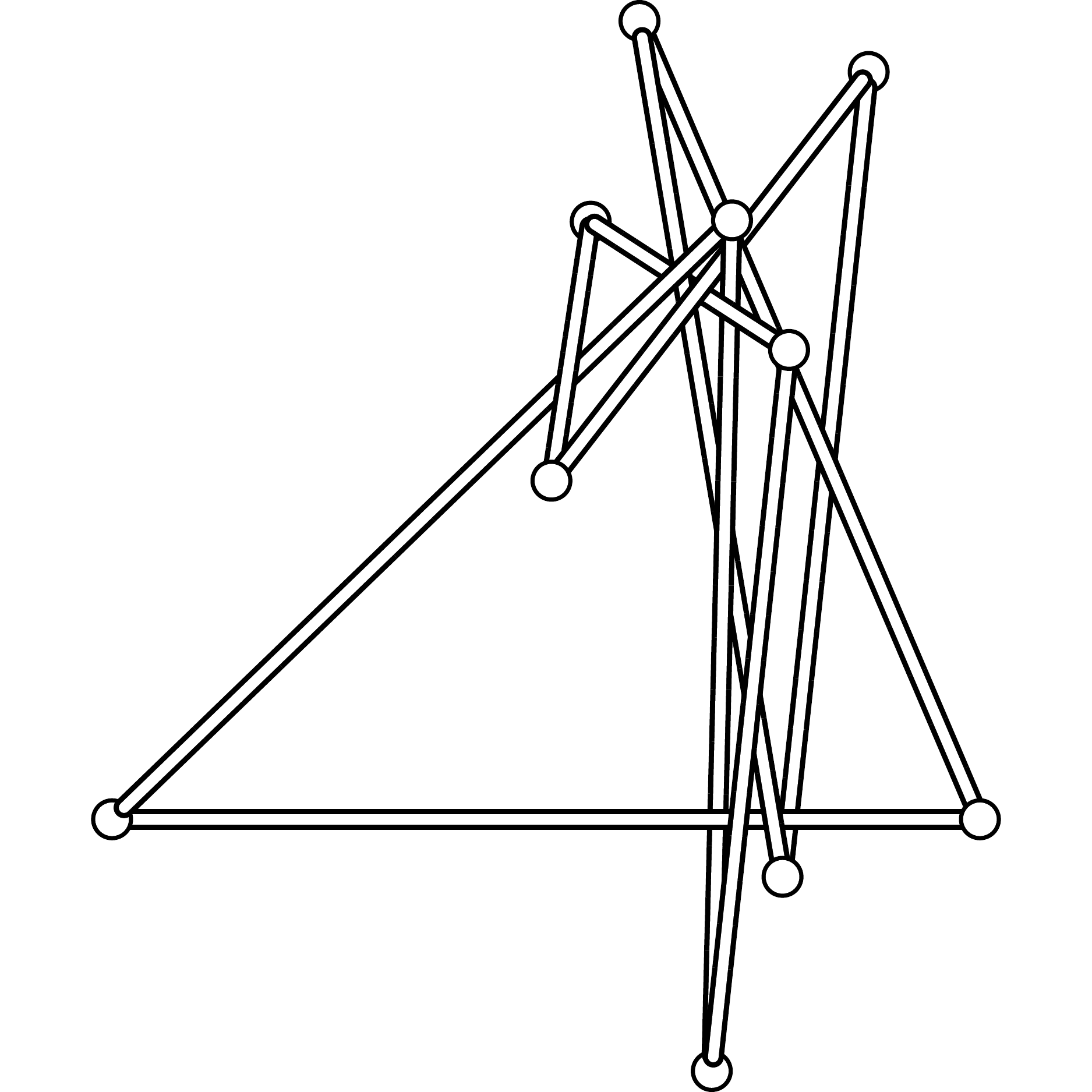}

\begin{tabular}{lll}

0. & 0. & 0. \\
0.9999999999999999 & 0. & 0. \\
0.23742406673458422 & 0.6468987138682376 & 0. \\
0.8624816316358974 & -0.04229570782933964 & -0.36648886703497885 \\
0.4011057049003684 & 0.5864895424110548 & 0.2594195634155591 \\
-0.1801631109472681 & -0.1253414478506597 & -0.13481793483405285 \\
0.6601561286258801 & 0.38286488530589013 & 0.0538349170998087 \\
-0.008962871699621913 & 0.35627600812192406 & -0.6888445528019091 \\
0.09031617352847202 & -0.09176640719170803 & 0.1996382355842952 \\
0.9044534650000824 & 0.4199643332833877 & -0.07479230187114154\\

\end{tabular}
\end{coordinates}

\end{multicols*}

\section{A 5-Superbridge Realization of $10_{37}$}
\label{sec:10_37}

Since $10_{37}$ is a 2-bridge knot and Adams et al.~\cite{Adams:2020vm} showed that $\superbridge[K] \leq 3\bridge[K]-1$, it follows that $\superbridge[10_{37}] \leq 5$, even though the best known bound on stick number is $\stick[10_{37}] \leq 12$~\cite{Rawdon:2002wj}. Adams et al.'s proof is based on putting the knot in a standard configuration that has superbridge number $3\bridge[K]-1$, so one could in principle follow their procedure to find a concrete realization of $10_{37}$ with superbridge number equal to 5. 

Rather than doing that, here is a $12$-stick realization of $10_{37}$ with superbridge number 5:

\medskip

\begin{tabular*}{0.85\textwidth}{C{2.2in} R{.4in} R{.4in} R{.4in} | R{1in}}
\multicolumn{5}{c}{$10_{37}$} \\
\cline{1-5}\noalign{\smallskip}
\multirow{12}{*}{\includegraphics[height=1.8in]{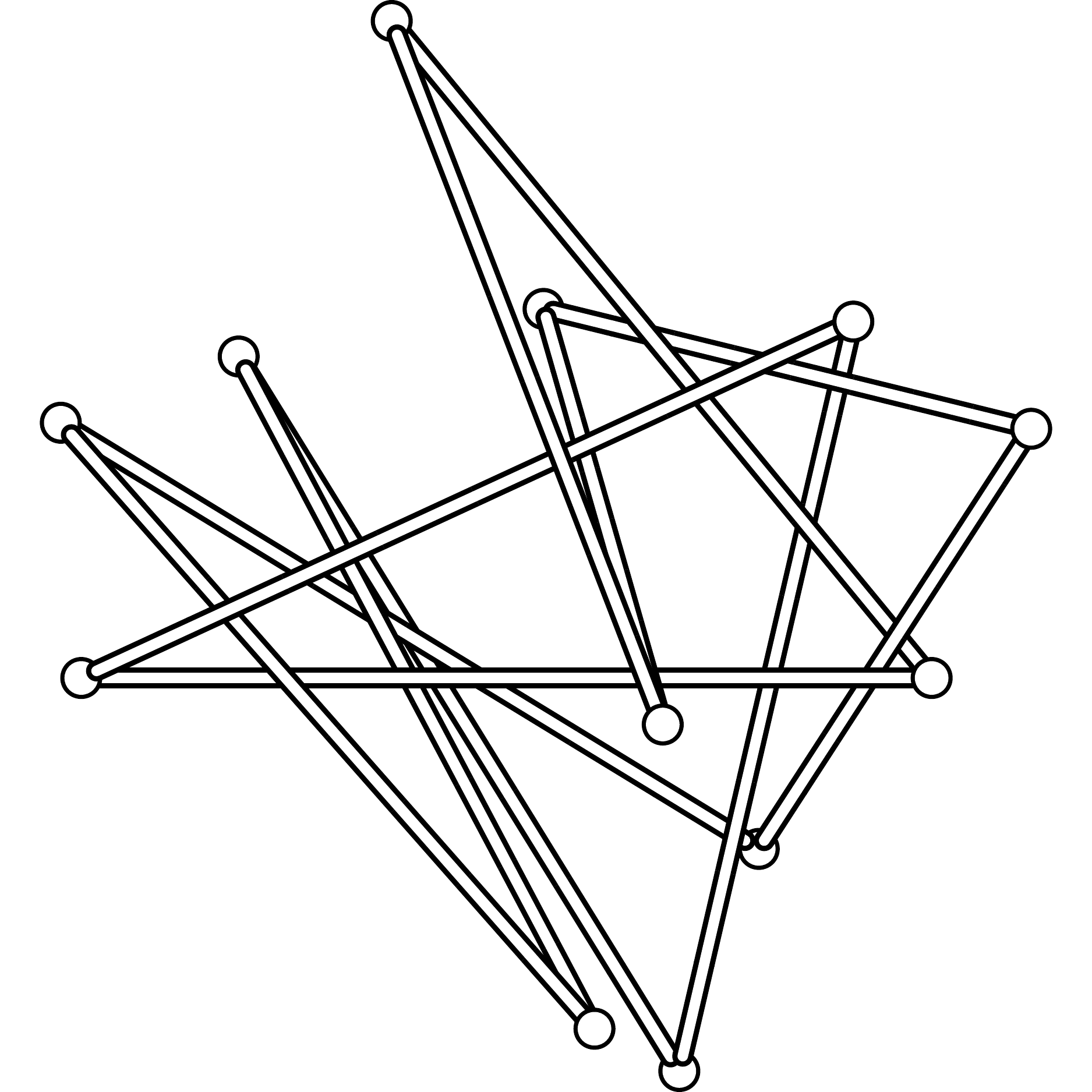}} & $0$ & $0$ & $0$ & $1$ \\
& $1000$ & $0$ & $0$ & $1$ \\
& $156$ & $536$ & $0$ & $1$ \\
& $655$ & $-107$ & $-580$ & $1$ \\
& $779$ & $-509$ & $327$ & $1$ \\
& $73$ & $-290$ & $-347$ & $20889658716$ \\
& $903$ & $-312$ & $211$ & $1$ \\
& $557$ & $-322$ & $-728$ & $1$ \\
& $592$ & $362$ & $1$ & $637602926$ \\
& $521$ & $-622$ & $162$ & $1683870031$ \\
& $381$ & $227$ & $-347$ & $116253911$ \\
& $847$ & $80$ & $525$ & $20181547124$ \\
\end{tabular*}

\medskip

The knot is shown in orthographic perspective from the direction of the positive $z$-axis relative to the vertex coordinates given in the three columns to the right of the image. These coordinates have been rounded to three significant digits and expressed as integers so that they are amenable to exact arithmetic. The last column gives the coordinates of a vector certifying that this specific $10_{37}$ has superbridge number $\leq 5$, using the approach from~\cite{Shonkwiler:2020gi}.

Specifically, a 12-stick realization of a knot $K$ has superbridge number equal to 6 only if there exists a direction $\vec{v} \in S^2$ so that the projection of each of the 12 vertices to the line spanned by $\vec{v}$ is either a local minimum or a local maximum. If $\vec{e}_i$ is the $i$th edge vector, this means we can find $\vec{v}$ so that $\vec{v} \cdot \vec{e}_1, \dots , \vec{v} \cdot \vec{e}_{12}$ alternates signs, or (after possibly replacing $\vec{v}$ with $-\vec{v}$) equivalently
\[
	\vec{v}^T \begin{bmatrix} \vec{e}_1 & -\vec{e}_2 & \cdots & \vec{e}_{11} & -\vec{e}_{12} \end{bmatrix}
\]
has all positive entries. By Gordan's theorem~\cite{Gordan:1873dz}, either such a $\vec{v}$ exists, or there exists $\vec{u} \in \R^{12}$ such that
\[
	\begin{bmatrix} \vec{e}_1 & -\vec{e}_2 & \cdots & \vec{e}_{11} & -\vec{e}_{12} \end{bmatrix} \vec{u} = \vec{0},
\]
but not both. The last column above gives such a $\vec{u}$ (found using \emph{Mathematica}'s {\tt FindInstance} function), which provides a certificate that the superbridge number of this realization of $10_{37}$ is strictly less than 6.

\clearpage

\bibliography{stickknots-special,stickknots}

\end{document}